\documentclass[11pt]{article}
\usepackage{amsmath,amssymb,amsthm,amscd}

\newtheorem{theorem}{Theorem}[section]
\newtheorem{lemma}[theorem]{Lemma}
\newtheorem{proposition}[theorem]{Proposition}
\newtheorem{corollary}[theorem]{Corollary}
\newtheorem{example}[theorem]{Example}

\theoremstyle{plain}

\theoremstyle{definition}
\newtheorem{definition}[theorem]{Definition}
\newtheorem{remark}[theorem]{Remark}

\numberwithin{equation}{section}

\renewcommand{\labelenumi}{\textup{(\theenumi)}}

\renewcommand{\phi}{\varphi}

\newcommand{\Homeo}{\operatorname{Homeo}}

\newcommand{\id}{\operatorname{id}}
\newcommand{\Ker}{\operatorname{Ker}}

\newcommand{\Ad}{\operatorname{Ad}}

\newcommand{\KSSEOA}{{K_0^{\operatorname{SSE}}({\mathcal{O}}_A)}}
\newcommand{\KSSEOB}{{K_0^{\operatorname{SSE}}({\mathcal{O}}_B)}}
\newcommand{\K}{\mathcal{K}}
\newcommand{\A}{\mathcal{A}}
\newcommand{\B}{\mathcal{B}}

\newcommand{\N}{\mathbb{N}}
\newcommand{\C}{\mathbb{C}}
\newcommand{\R}{\mathbb{R}}
\newcommand{\T}{\mathbb{T}}
\newcommand{\Z}{\mathbb{Z}}
\newcommand{\Zp}{{\mathbb{Z}}_+}

\def\OA{{{\mathcal{O}}_A}}
\def\DA{{{\mathcal{D}}_A}}
\def\OB{{{\mathcal{O}}_B}}
\def\DB{{{\mathcal{D}}_B}}

\title{Topological conjugacy of topological Markov shifts and Cuntz--Krieger algebras}
\author{Kengo Matsumoto \\
Department of Mathematics \\
Joetsu University of Education \\
Joetsu, 943-8512, Japan
}
\date{}

\begin{document}
\maketitle

\def\det{{{\operatorname{det}}}}

%\maketitle
\begin{abstract}
For  an irreducible non-permutation matrix $A$,
the triplet
$(\OA,\DA,\rho^A)$ 
for  the Cuntz-Krieger algebra $\OA$, 
its canonical maximal abelian $C^*$-subalgebra $\DA$, 
and its gauge action $\rho^A$
is called the Cuntz--Krieger triplet.
We introduce a notion of strong Morita equivalence in the Cuntz--Krieger triplets,
and prove that 
two Cuntz--Krieger triplets
$(\OA,\DA,\rho^A)$ and
$(\OB,\DB,\rho^B)$ 
are strong Morita equivalent 
if and only if
$A$ and $B$ are strong shift equivalent.
 We also show that
the generalized gauge actions on 
the stabilized Cuntz--Krieger algebras are cocycle conjugate if the underlying
matrices are strong shift equivalent.
By clarifying K-theoretic behavior of the cocycle conjugacy, 
we  investigate a relationship between 
cocycle conjugacy of the gauge actions on the stabilized Cuntz--Krieger algebras
and 
topological conjugacy of the underlying topological Markov shifts.
\end{abstract}

%{\it Mathematics Subject Classification}:
% Primary 46L55; Secondary 46L35, 37B10.

%{\it Keywords and phrases}:
%Topological Markov shifts, strong shift equivalence,  Cuntz--Krieger algebras,
%K-theory, gauge action

%%%%%%%%%%%%%%%%%%%%%%%%%%%%%%%%%%%%%%%%%%%%%%%%%%%%   

\def\OA{{{\mathcal{O}}_A}}
\def\OB{{{\mathcal{O}}_B}}
\def\OZ{{{\mathcal{O}}_Z}}
\def\OTA{{{\mathcal{O}}_{\tilde{A}}}}
\def\SOA{{{\mathcal{O}}_A}\otimes{\mathcal{K}}}
\def\SOB{{{\mathcal{O}}_B}\otimes{\mathcal{K}}}
\def\SOZ{{{\mathcal{O}}_Z}\otimes{\mathcal{K}}}
\def\SOTA{{{\mathcal{O}}_{\tilde{A}}\otimes{\mathcal{K}}}}
\def\DA{{{\mathcal{D}}_A}}
\def\DB{{{\mathcal{D}}_B}}
\def\DZ{{{\mathcal{D}}_Z}}
\def\DTA{{{\mathcal{D}}_{\tilde{A}}}}
\def\SDA{{{\mathcal{D}}_A}\otimes{\mathcal{C}}}
\def\SDB{{{\mathcal{D}}_B}\otimes{\mathcal{C}}}
\def\SDZ{{{\mathcal{D}}_Z}\otimes{\mathcal{C}}}
\def\SDTA{{{\mathcal{D}}_{\tilde{A}}\otimes{\mathcal{C}}}}
\def\BC{{{\mathcal{B}}_C}}
\def\BD{{{\mathcal{B}}_D}}
\def\OAG{{\mathcal{O}}_{A^G}}
\def\OBG{{\mathcal{O}}_{B^G}}
\def\Max{{{\operatorname{Max}}}}
\def\Per{{{\operatorname{Per}}}}
\def\PerB{{{\operatorname{PerB}}}}
\def\Homeo{{{\operatorname{Homeo}}}}
\def\HA{{{\frak H}_A}}
\def\HB{{{\frak H}_B}}
\def\HSA{{H_{\sigma_A}(X_A)}}
\def\Out{{{\operatorname{Out}}}}
\def\Aut{{{\operatorname{Aut}}}}
\def\Ad{{{\operatorname{Ad}}}}
\def\Inn{{{\operatorname{Inn}}}}
\def\det{{{\operatorname{det}}}}
\def\exp{{{\operatorname{exp}}}}
\def\cobdy{{{\operatorname{cobdy}}}}
\def\Ker{{{\operatorname{Ker}}}}
\def\ind{{{\operatorname{ind}}}}
\def\id{{{\operatorname{id}}}}
\def\supp{{{\operatorname{supp}}}}
\def\co{{{\operatorname{co}}}}
\def\Sco{{{\operatorname{Sco}}}}
\def\ActA{{{\operatorname{Act}_{\DA}(\mathbb{T},\OA)}}}
\def\ActB{{{\operatorname{Act}_{\DB}(\mathbb{T},\OB)}}}
\def\U{{{\mathcal{U}}}}
%%%%%%%%%%%%%%%%%%%%%%%%%%%%%%%%%%%%%%%

%%%%%%%%%%%%%%%%%%%%%%%%%%%%
%%%%%%%%%%%%%%%%%%%%%%%%%%%%%%%%%%%%%%%
\section{Introduction and Preliminaries}
%%%%%%%%%%%%%%%%%%%%%%%%%%%%%%%%%%%%%%%%%%%%
%%%%%%%%%%%%%%%%%%%%%%%%%%%%%%%%%%%%%%%%%
Let
$A=[A(i,j)]_{i,j=1}^N$
be an irreducible matrix  
with entries in $\{0,1\}$ with $1<N\in \N$.
We assume that $A$ is not any permutation matrix. 
In \cite{CK}, J. Cuntz and W. Krieger 
have introduced a $C^*$-algebra 
$\OA$ associated to topological Markov shift $(X_A,\sigma_A)$.
The $C^*$-algebra is called the Cuntz--Krieger algebra,
which is a universal unique purely infinite simple $C^*$-algebra generated by
partial isometries
$S_1,\dots,S_N$
subject to the relations:
\begin{equation} 
\sum_{j=1}^N S_j S_j^* = 1, \qquad
S_i^* S_i = \sum_{j=1}^N A(i,j) S_jS_j^*, 
\quad i=1,\dots,N. \label{eq:CK}
\end{equation} 
For $t \in {\mathbb{R}}/\Z = {\mathbb{T}}$,
the correspondence
$S_i \rightarrow e^{2 \pi\sqrt{-1}t}S_i,
\, i=1,\dots,N$
gives rise to an automorphism
of $\OA$ denoted by
$\rho^A_t$.
The automorphisms
$\rho^A_t, t \in {\mathbb{T}}$
yield an action of ${\mathbb{T}}$
on $\OA$ called the gauge action.
Cuntz and Krieger in \cite{CK} have shown that the algebra 
$\OA$ has close relationships with the underlying dynamical system 
called topological Markov shift.
Let us denote by
$X_A$ 
 the shift space 
\begin{equation}
X_A = \{ (x_n )_{n \in \N} \in \{1,\dots,N \}^{\N}
\mid
A(x_n,x_{n+1}) =1 \text{ for all } n \in {\N}
\}. \label{eq:onMarkovshift}
\end{equation}
Define the shift transformation $\sigma_A$ on $X_A$
by 
$\sigma_{A}((x_n)_{n \in {\N}})=(x_{n+1} )_{n \in \N}$,
which is a continuous surjection on $X_A$.
The topological dynamical system $(X_A,\sigma_A)$
is called the one-sided topological Markov shift for matrix $A$.
The two-sided topological Markov shift 
$(\bar{X}_A, \bar{\sigma}_A)$
is similarly defined with the shift space
\begin{equation}
\bar{X}_A = \{ (x_n )_{n \in \Z} \in \{1,\dots,N \}^{\Z}
\mid
A(x_n,x_{n+1}) =1 \text{ for all } n \in {\Z}
\} \label{eq:twoMarkovshift}
\end{equation}
and 
the shift homeomorphism
$\bar{\sigma}_{A}((x_n)_{n \in {\Z}})=(x_{n+1} )_{n \in \Z}$
on 
$\bar{X}_A$.

 Let  us denote by 
$\DA$ 
the $C^*$-subalgebra of $\OA$
generated by the projections of the form:
$S_{i_1}\cdots S_{i_n}S_{i_n}^* \cdots S_{i_1}^*,
i_1,\dots,i_n =1,\dots,N$.
The subalgebra $\DA$ is
 canonically isomorphic to the commutative 
$C^*$-algebra $C(X_A)$
of the complex valued continuous functions on $X_A$
by identifying the projection
$S_{i_1}\cdots S_{i_n}S_{i_n}^* \cdots S_{i_1}^*
$
with the characteristic function
$\chi_{U_{i_1\cdots i_n}} \in C(X_A)$
of the cylinder set
$U_{i_1\cdots i_n}$
for the word
${i_1\cdots i_n}$.
Let us denote by $\K$ the $C^*$-algebra $\K(\ell^2(\N))$ of compact operators on a separable infinite dimensional  Hilbert space  $\ell^2(\N)$ and 
by ${\mathcal{C}}$ its maximal abelian $C^*$-subalgebra of diagonal operators.

In \cite{Williams}, R. F. Williams  proved that 
the topological Markov shifts 
$(\bar{X}_A, \bar{\sigma}_A)$ and 
$(\bar{X}_B, \bar{\sigma}_B)$ 
are topologically conjugate
if and only if the matrices $A,B$ are strong shift equivalent.
Two nonnegative matrices $A, B$ are said to be  elementary equivalent
if there exist
nonnegative rectangular matrices $C,D$ such that 
$A = CD, B = DC$.
We write it as $A\underset{C,D}{\approx}B$.
If there exists a finite sequence of nonnegative matrices $A_0,A_1,\dots, A_n$
such that
$A = A_0, B = A_n$ and $A_i$ is elementary equivalent to $A_{i+1}$ for 
$i=1,2,\dots, n-1$, 
then $A$ and $B$ are said to be strong shift equivalent.  
 Hence elementary equivalence generates topological conjugacy of two-sided topological Markov shifts. 

%%%%%%%%%%%%%%%%%%%%%%%%%%%%%%%%%%%%%%%%%%%%%%%%%
Let $A$ be an irreducible non-permutation matrix.
The triplet
$(\OA,\DA,\rho^A)$ 
for the Cuntz-Krieger algebra $\OA$, 
its canonical maximal abelian $C^*$-subalgebra $\DA$, 
and its gauge action $\rho^A$
is called the Cuntz--Krieger triplet for the matrix $A$. 
As pointed out in \cite{MaETDS2004},
two elementary equivalence matrices 
$A = CD, B = DC$ yield
$\OA-\OB$-imprimitivity bimodule via Cuntz--Krieger algebra
$\OZ$ for the matrix $Z$ defined by
$Z =
\begin{bmatrix}
0 & C \\
D & 0
\end{bmatrix}.
$

In the first part of the paper,
We will introduce a notion of strong Morita equivalence in the Cuntz--Krieger triplets,
and prove the following theorem.
\begin{theorem}[Corolary \ref{cor:SMECK}]
The Cuntz--Krieger triplets
$(\OA,\DA,\rho^A)$ and
$(\OB,\DB,\rho^B)$ 
are strong Morita equivalent 
if and only if the matrices
$A$ and $B$ are strong shift equivalent.
\end{theorem}
It is well-known that two unital $C^*$-algebras 
$\A$ and $\B$
are strong Morita equivalent 
if and only if their stabilizations 
$\A\otimes\K$ and $\B\otimes\K$ 
are isomorphic  by Brown--Green--Rieffel Theorem \cite[Theorem 1.2]{BGR} 
(cf. \cite{BGR}, \cite{Combes}).
We will next study relationships between stabilized Cuntz--Krieger algebras
with their gauge actions
and strong shift equivalence matrices.
We must emphasize that Cuntz and Krieger in \cite[3.8 Theorem]{CK} and Cuntz in \cite[2.3 Theorem]{Cu3}
 have  shown that the stabilized 
Cuntz--Krieger triplet
$(\SOA,\SDA,\rho^A\otimes\id)$
is invariant under topological conjugacy of the two-sided topological Markov shifts
$(\bar{X}_A,\bar{\sigma}_A)$. 
 We will investigate stabilizations of generalized gauge actions from a view point of 
flow equivalence.

%%%%%%%%%%%%%%%%%%%%%%%%%%%%%%%%%%%%%%%%%%%%%%%%%

Let us denote by 
$C(X_A, \Z)$ the set of $\Z$-valued continuous functions on $X_A$.
For $f \in C(X_A, \Z)$,
define a one-parameter unitary group
$U_t(f), t \in \T= \R/\Z$ in $\DA$
by 
\begin{equation}
U_t(f) = \exp({2\pi \sqrt{-1} t f}), \label{eq:utf}
\end{equation}
and an automorphism
$\rho_t^{A,f}$ on $\OA$ for each $t \in \T$ 
by 
\begin{equation}
\rho_t^{A,f}(S_i) = U_t(f) S_i, \qquad i=1,\dots,N. \label{eq:rhotf}
\end{equation}
For $f \equiv 1$, the action 
$\rho_t^{A,1}$ is the gauge action denoted by $\rho_t^A$.
Suppose that 
$A = CD$ and $B =DC$ for some nonnegative rectangular matrices $C,D$.
Then there exist  homomorphisms
$\phi:C(X_A,\Z) \rightarrow  C(X_B,\Z)$
and
$\psi:C(X_B,\Z) \rightarrow  C(X_A,\Z)$
such that 
\begin{equation}
(\psi \circ \phi)(f) = f \circ \sigma_A,\qquad
(\phi \circ \psi)(g) = g \circ \sigma_B
\end{equation}
for $f \in C(X_A,\Z)$ and $g \in C(X_B,\Z)$.
Let us denote by $(H^A, H^A_+)$
the ordered cohomology groups 
 for the one-sided topological Markov shift 
$(X_A,\sigma_A)$
which has been introduced in \cite{MMKyoto} by setting
\begin{equation*}
H^A=C(X_A,{\mathbb{Z}})/\{\eta - \eta\circ\sigma_A\mid\eta\in C(X_A,{\mathbb{Z}})\}
\end{equation*}
and its positive cone
\begin{equation*}
H^A_+=\{[\eta]\in H^A\mid\eta(x)\geq0 \text{ for all } x\in X_A\}. 
\end{equation*}
The ordered cohomology group 
$(\bar{H}^A, \bar{H}^A_+)$ for $(\bar{X}_A,\bar{\sigma}_A)$ 
has been considered  by Y. T. Poon in \cite{Po}.
The latter ordered group
$(\bar{H}^A, \bar{H}^A_+)$
has been proved to be a complete invariant of flow equivalence of 
the two-sided topological Markov shift  $(\bar{X}_A,\bar{\sigma}_A)$ 
by M. Boyle and D. Handelman in \cite{BH}. 
The two ordered groups 
$(\bar H^A,\bar H^A_+)$ and $(H^A,H^A_+)$ are actually isomorphic 
(\cite[Lemma 3.1]{MMKyoto}).

%%%%%%%%%%%%%%%%%%%%%%%%%%%%%%%%%
In \cite{MaMZ2016},
the following result has been proved.
\begin{theorem}[{\cite[Corollary 4.4]{MaMZ2016}}]
Suppose that 
$A$ and $B$ are strong shift equivalent.
Then there exist an isomorphism
$\Phi:\SOA \rightarrow \SOB$ satisfying 
$\Phi(\SDA) = \SDB$
and a homomorphism
$\phi:C(X_A,\Z) \rightarrow C(X_B,\Z)$
of ordered groups
which induces an isomorphism
between
$(H^A, H^A_+)$ and $(H^B, H^B_+)$ of ordered groups
such that
for each function
$f\in C(X_A,\Z)$
there exists a unitary one-cocycle
$u_t^f \in \U(M(\SOA))$  
relative to 
$\rho^{A,f} \otimes\id$
satisfying 
\begin{equation*}
\Phi \circ \Ad(u_t^f) \circ (\rho^{A,f}_t\otimes \id)
 = (\rho^{B,\phi(f)}_t\otimes\id) \circ \Phi
\quad
\text{ for }
 t \in \T.
\end{equation*}
\end{theorem}
In the second part of the present paper,
we will study K-theoretic behavior of the above isomorphism
$\Phi:\SOA \rightarrow \SOB$. 
Let us denote by 
$\epsilon_A: K_0(\OA) \rightarrow \Z^N/{(\id - A^{t})\Z^N}$
the isomorphism defined in \cite[3.1 Proposition]{Cu3}
satisfying 
$\epsilon_A([1_A]) = [(1,1,\dots,1)]$,
where $1_A$ is the unit of $\OA$.
We will prove the following theorem.
\begin{theorem}[Proposition \ref{prop:main1} and Theorem \ref{thm:KC}]
Suppose that $A$ and $B$ are elementary equivalent such that
$A = CD$ and $B =DC$.
Then there exist an isomorphism
$\Phi:\SOA \rightarrow \SOB$ satisfying 
$\Phi(\SDA) = \SDB$ 
and a unitary representation 
$t \in \T \rightarrow u^f_t \in M(\SDA)$ for each $f \in C(X_A,\Z)$
such that
\begin{equation*}
\Phi  \circ \Ad(u_t^f) \circ (\rho^{A,f}_t\otimes \id)
 = (\rho^{B,\phi(f)}_t\otimes\id) \circ \Phi
\quad
\text{ for }
f \in C(X_A,\Z), \, 
 t \in \T
\end{equation*}
and the diagram 
$$
\begin{CD}
K_0(\OA) @>\Phi_* >> K_0(\OB) \\
@V{\epsilon_A }VV  @VV{\epsilon_B}V \\
\Z^N/{(\id - A^{t})\Z^N} @> \Phi_{C^t} >> \Z^M/{(\id - B^{t})\Z^M} 
\end{CD}
$$
is commutative,
where $\Phi_{C^t}$ is an isomorphism induced by multiplying the matrix
$C^t$. 
%and $\epsilon_A: K_0(\OA) \rightarrow \Z^N/{(\id - A^{t})\Z^N}$
%is an isomorphism defined by
%$\epsilon_A([S_iS_i^*] ) = [e_i]$ the class of the vector $e_i$ in $\Z^N$
%whose $i$th component is one, and zero elsewhere. 
\end{theorem}

In the third part of the paper, we will study the converse of 
the above theorem for the gauge actions.
We will introduce an invariant $\KSSEOA$
which is a non-empty subset of $K_0(\OA)$. 
The invariant 
$\KSSEOA$
is realized  
as a subset of $\Z^N /(\id - A^t)\Z^N$  consisting of the classes $[v]$
of vectors 
$v \in \Z^N$
such that 
$v = D_1^t \cdots D_{n-1}^t D_n^t [1,1,\dots,1]^t
$
for some strong shift equivalences
$
A \underset{C_1,D_1}{\approx}\cdots \underset{C_{n},D_{n}}{\approx}D_n C_n
$
(Proposition \ref{prop:algSSE}).
We will then prove the following theorem.
\begin{theorem}[Theorem \ref{thm:main2}]
Let $A, B$ be irreducible and non-permutation matrices.
The following two assertions are equivalent.
\begin{enumerate}
\renewcommand{\theenumi}{\roman{enumi}}
\renewcommand{\labelenumi}{\textup{(\theenumi)}}
\item
Two-sided topological Markov shifts 
$(\bar{X}_A, \bar{\sigma}_A)$ and 
$(\bar{X}_B, \bar{\sigma}_B)$ 
are topologically conjugate.
\item
There exist an isomorphism
$\Phi:\SOA \rightarrow \SOB$ 
and  a unitary representation 
$t \in \T \rightarrow u_t^A \in M(\SDA)$ 
such that 
\begin{align*}
\Phi(\SDA) = \SDB, \qquad &
\Phi \circ \Ad(u_t^A) \circ  (\rho^{A}_t\otimes \id)
 = (\rho^{B}_t\otimes\id) \circ \Phi
\text{ for }
 t \in \T, \\
\Phi_*(\KSSEOA) &  = \KSSEOB.
\end{align*}
\end{enumerate}
\end{theorem}
We say that $A$ has {\it full units} if 
$ \KSSEOA = K_0(\OA)$.
The condition  $ \KSSEOA = K_0(\OA)$ is able to describe in terms of the matrix $A$
as in Proposition \ref{prop:algSSE}.
%Hence we have a partial answer for the question 10  raised in the Graph Algebra Problem Page in Mark Tomforde's web page \cite{TomfordProblem}.
\begin{corollary}[Corollary \ref{cor:main3}]
Suppose that matrices $A$ and $B$ have full units.
Then two-sided topological Markov shifts
$(\bar{X}_A, \bar{\sigma}_A)$ and
$(\bar{X}_B, \bar{\sigma}_B)$
are topologically conjugate if and only if there exist an isomorphism 
$\Phi:\SOA \rightarrow \SOB $ of $C^*$-algebras 
and 
 a unitary representation 
$t \in \T \rightarrow u_t^A \in M(\SDA)$ 
such that
\begin{equation*}
\Phi(\SDA) = \SDB, \qquad
\Phi \circ \Ad(u_t^A) \circ (\rho^{A}_t \otimes\id) 
=  (\rho^{B}_t \otimes \id) \circ \Phi. \\
\end{equation*}
\end{corollary}
\medskip
Throughout the paper,
we denote by $\N$
the set of positive integers
and by $\Zp$
the set of nonnegative integers,
respectively.
For one-sided topological Markov shift
$(X_A,\sigma_A)$,
a word $\mu =(\mu_1, \dots, \mu_k)$ 
for $\mu_i \in \{1,\dots,N\}$
is said to be admissible for $X_A$ 
if  
$(\mu_1, \dots, \mu_k) =(x_1, \dots, x_k)$
for some element
$(x_n)_{n \in \N} \in X_A$.
The length of $\mu$ is denoted by $|\mu| =k$.
 We denote by $B_k(X_A)$ the set of all admissible words of length $k$.
We similarly denote by $B_k(\bar{X}_A)$
the set of admissible words of length $k$,
so that   $B_k(\bar{X}_A) = B_k(X_A).$
The cylinder set 
$\{ (x_n )_{n \in \N} \in X_A 
\mid x_1 =\mu_1,\dots, x_k = \mu_k \}$
for $\mu=(\mu_1,\dots,\mu_k) \in B_k(X_A)$
is denoted by $U_\mu$.  

\medskip
This paper is a second revised version of arXiv:1604.02763v1, in which the given proofs of the main results were incorrect. 

%%%%%%%%%%%%%%%%%%%%%%%%%%%%%%%%%%%%%

%%%%%%%%%%%%%%%%%%%%%%%%%%%%%%%%%%%%%%%%%%%%%%%%%%%%%%%%
%%%%%%%%%%%%%%%%%%%%%%%%%%%%%%%%%%%%%%%%%%%%%%%%%%%%%%%%
\section{Strong Morita equivalence for Cuntz--Krieger triplets}
%%%%%%%%%%%%%%%%%%%%%%%%%%%%%%%%%%%%%%%%%%%%%%%%%%%%%%%%
%%%%%%%%%%%%%%%%%%%%%%%%%%%%%%%%

There is a standard method to associate a Cuntz--Krieger algebra from a square matrix with entries in nonnegative integers as described in \cite[Section 4]{Ro}. 
Now we suppose that  $A =[A(i,j)]_{i,j=1}^N $ 
is an $N\times N$ matrix with entries in nonnegative integers.
Then the associated graph $G_A = (V_A,E_A)$ consists of the vertex set 
$V_A =\{ v^A_1, \dots, v_N^A\}$ of $N$ vertices and the edge set 
$E_A =\{a_1,\dots,a_{N_A} \}$, where 
there  are $A(i,j)$ edges  from $v_i^A$ to $v_j^A$.
 Hence the total number of edges is $\sum_{i,j=1}^N A(i,j)$ denoted by $N_A$.
For $a_i \in E_A$, 
denote by $t(a_i),  s(a_i)$ the terminal vertex  of $a_i$, 
the source vertex of $a_i$, respectively. 
The graph $G_A$ has the $N_A \times N_A$
 transition matrix $A^G =[A^G(i,j)]_{i,j=1}^{N_A}$ of edges 
defined by 
\begin{equation}
A^G(i, j) =
\begin{cases} 
 1 &  \text{  if  } t(a_i) = s(a_j), \\
 0 & \text{  otherwise}.
\end{cases} \label{eq:AG}
\end{equation}
The Cuntz--Krieger algebra $\OA$ for the matrix $A$ with
 entries in nonnegative integers  is defined as the Cuntz--Krieger algebra
 ${\mathcal{O}}_{A^G}$ 
for the matrix $A^G$ which is the universal $C^*$-algebra generated by 
partial isometries
$S_{a_i}$ indexed by edges $a_i, i=1,\dots, N_A$ subject to the relations:
\begin{equation}
\sum_{j=1}^{N_A}  S_{a_j} S_{a_j}^* = 1, 
\qquad
S_{a_i}^* S_{a_i} = 
\sum_{j=1}^{N_A}  A^G(i, j ) S_{a_j} S_{a_j}^*
 \quad \text{ for } i=1,\dots,N_A.
\label{eq:OAG}
\end{equation}
For a word 
$\mu =(\mu_1, \dots, \mu_k), \mu_i  \in E_Z$,
we denote by $S_\mu$ the partial isometry
$S_{\mu_1}\cdots S_{\mu_k}$.

%%%%%%%%%%%%%%%%%%%%%%%%%%%%%%%%%%%%%%%%

As in the standard text books \cite{Kitchens},  \cite{LM}
of symbolic dynamics, 
the two-sided topological Markov shift
defined by a square matrix with entries in $\{0,1\}$
is naturally topologically conjugate to 
a topological Markov shift of the edge shift defined by 
the underlying directed graph. 
 In what follows,
we consider edge shifts and hence  square matrices
with entries in nonnegative integers 
(cf.  \cite{Kitchens}, \cite{LM}, \cite{Williams}, etc.).
Such a matrix is simply called  a nonnegative square matrix.
For a nonnegative square matrix $A$,
the two-sided shift space
$\bar{X}_A$  is defined by the two-sided shift space $\bar{X}_{A^G}$
for the matrix $A^G$ 
which consists of two-sided bi-infinite sequences of concatenated edges of
the directed graph $G_A$.

%%%%%%%%%%%%%%%%%%%%%%%%%%%%
Suppose that two nonnegative square matrices $A$ and $B$ 
are elementary equivalent
such that $A =CD$ and $B = DC$.
The sizes of the matrices
$A$ and $B$ are denoted by
$N$ and $M$ respectively,
so that 
$C$ is an $N \times M$ matrix and 
$D$ is an $M \times N$ matrix, respectively. 
 We set the square matrix
$
Z =
\begin{bmatrix}
0 & C \\
D & 0
\end{bmatrix}
$
as a block matrix, and we see
\begin{equation*}
Z^2 =
\begin{bmatrix}
CD & 0 \\
0 & DC
\end{bmatrix}
=
\begin{bmatrix}
A & 0 \\
0 & B
\end{bmatrix}.
\end{equation*}
Similarly to the directed graph 
$
G_A =(V_A,E_A),
$
let  us denote by
$
G_B =(V_B,E_B),
G_C =(V_C,E_C),
G_D =(V_D,E_D)
$
and
$
G_Z =(V_Z,E_Z)
$
the associated directed graphs to the nonnegative matrices
$ B, C, D$ and $Z$, respectively.  
By the equalities $A=CD$ and $B = DC$,
we may take bijections
$\varphi_{A,CD}$ from $E_A$ to a subset of $E_C \times E_D$
and
$\varphi_{B,DC}$ from $E_B$ to a subset of $E_D \times E_C$.
Let 
$ S_c, S_d, c \in E_C, d \in E_D$ 
be the generating partial isometries of 
the Cuntz--Krieger algebra 
$\OZ$
for the matrix $Z$, so that
$
\sum_{c \in E_C} S_c S_c^*
+
\sum_{d \in E_D} S_d S_d^*
=1$
and
\begin{equation*}
S_c^* S_c = \sum_{d \in E_D}Z(c,d) S_d S_d^*, \qquad
S_d^* S_d = \sum_{c \in E_C}Z(d,c) S_c S_c^*
\end{equation*}
for $c \in E_C, d \in E_D$.
Since 
$S_c S_d \ne 0$ 
(resp. $S_d S_c \ne 0$) 
if and only if 
$\varphi_{A,CD}(a) = cd$
(resp.
$\varphi_{B,DC}(b) = dc$)
 for some $a \in E_A$
(resp. $b \in E_B$),
we may identify $cd$ (resp. $dc$) 
with $a$ (resp. $b$)
through the map 
$\varphi_{A,CD}$ (resp. $\varphi_{B,DC}$).
We may then write 
$S_{cd} = S_a$
(resp. $S_{dc} = S_b$)
where $S_{cd}$ denotes $S_c S_d$ (resp. $S_{dc}$ denotes $S_d S_c$).
We define two particular projections $P_C$ and $P_D$ in $\DZ$ by 
$P_C = \sum_{c \in E_C} S_c S_c^*$
and
$P_D = \sum_{d \in E_D} S_d S_d^*$
so that $P_C + P_D =1$.
It has been shown in \cite{MaETDS2004} (cf. \cite{MaMZ2016})
that 
\begin{equation}
P_C \OZ P_C = \OA, \qquad P_D \OZ P_D = \OB,
\qquad
 \DZ P_C = \DA, \qquad  \DZ P_D = \DB. \label{eq:4.1}
\end{equation}
As in \cite[Lemma 3.1]{MaETDS2004},
both projections $P_C$ and $P_D$ are full projections so that  
$P_C \OZ P_D$ has a natural structure of 
$\OA-\OB$ imprimitivity bimodule that makes $\OA$ and $\OB$ 
strong Morita equivalent (cf. \cite{Rieffel1}, \cite{Rieffel2}).

Let
$\rho^Z, \rho^A, \rho^B$
be the gauge actions of 
$\OZ, \OA,\OB$, respectively.
Since $S_c S_d$ (resp. $S_d S_c$)
in $\OZ$ is identified with
$S_a$ in $\OA$ (resp. $S_b$ in $\OB$)
if
$\varphi_{A,CD}(a) = cd$ (resp. $\varphi_{B,DC}(b) = dc$,
we have
\begin{equation}
\rho^Z_t |_{P_C \OZ P_C} = \rho^A_{2t} \text{ on } \OA,
\qquad
\rho^Z_t |_{P_D \OZ P_D} = \rho^B_{2t} \text{ on } \OB. \label{eq:gaugeZAB}
\end{equation}

%%%%%%%%%%%%%%%%%%%%%%%%%%%%%
%%%%%%%%%%%%%%%%%%%%%%%%%%%%

\medskip

Let $A$ be an irreducible non-permutation matrix.
The triplet
$(\OA,\DA,\rho^A)$ 
for the Cuntz-Krieger algebra $\OA$, 
its canonical maximal abelian $C^*$-subalgebra $\DA$, 
and its gauge action $\rho^A$
is called the Cuntz--Krieger triplet for the matrix $A$. 
In this section we will define the notion of strong Morita equivalence in 
Cuntz--Krieger triplets.
We will then prove that 
the Cuntz--Krieger triplets
$(\OA,\DA,\rho^A)$ and
$(\OB,\DB,\rho^B)$ 
are strong Morita equivalent 
if and only if the matrices
$A$ and $B$ are strong shift equivalent.
Let $A, B$ be irreducible non-permutation matrices.
\begin{definition}\label{def:SME}
The Cuntz--Krieger triplets
$(\OA,\DA,\rho^A)$ and
$(\OB,\DB,\rho^B)$ 
are said to be {\it strong Morita equivalent in} $1$-{\it step}
if 
there exist a Cuntz--Krieger triplet 
$(\OZ,\DZ,\rho^Z)$ for some nonnegative matrix $Z$
and projections 
$P_A, P_B  \in \DZ$ having  the following properties:
\begin{enumerate}
\item $P_A + P_B = 1,$
\item $ P_A \OZ P_A = \OA$ and $P_B \OZ P_B = \OB,$
\item $\DZ P_A = \DA$ and $\DZ P_B =\DB,$
\item $\rho^Z_t |_{P_A \OZ P_A} = \rho^A_{2t}$ on $\OA$
  and $\rho^Z_t |_{P_B \OZ P_B} = \rho^B_{2t}$ on $\OB$ for $t \in \T$.
\end{enumerate}
\end{definition}
In this case, we say that 
$(\OA,\DA,\rho^A)$ and
$(\OB,\DB,\rho^B)$ 
are strong Morita equivalent in $1$-step
via 
$(\OZ,\DZ,\rho^Z)$.
If two  Cuntz--Krieger triplets
$(\OA,\DA,\rho^A)$ and
$(\OB,\DB,\rho^B)$ are connected through $n$-chains of 
strong Morita equivalences in $1$-step, 
$(\OA,\DA,\rho^A)$ and
$(\OB,\DB,\rho^B)$ are said to be 
strong Morita equivalent in $n$-step,
or simply, 
strong Morita equivalent.

We note that if there exists an isomorphism
$\Phi:\OA \longrightarrow \OB$ satisfying
$\Phi(\DA) = \DB$ and 
$\Phi\circ\rho^A_t =\rho^B_t\circ \Phi, \, t \in \T,$
then the one-sided topological Markov shifts
$(X_A,\sigma_A)$ and $(X_B,\sigma_B)$ 
are eventually conjugate (\cite[Corollary 3.5]{MaMZ2016}),
so that their two-sided topological Markov shifts
$(\bar{X}_A, \bar{\sigma}_A)$ and $(\bar{X}_B, \bar{\sigma}_B)$ 
are topologically conjugate by \cite[Theorem 5.5]{MaJOT2015} 
(cf. \cite[Theorem 6.7]{MaJOT2015}),
and hence the matrices $A$ and $B$ are strong shift equivalent.  

\begin{proposition}\label{prop:SME1}
If $A$ and $B$ are elementary equivalent, 
then their   Cuntz--Krieger triplets
$(\OA,\DA,\rho^A)$ and
$(\OB,\DB,\rho^B)$ are 
strong Morita equivalent in $1$-steps,
\end{proposition}
\begin{proof}
Suppose that  $A$ and $B$ are elementary equivalent
such that $A = CD, B= DC$.
Let $Z$ 
be  the square matrix
$
Z =
\begin{bmatrix}
0 & C \\
D & 0
\end{bmatrix}.
$
By the above discussions, there exist projections
$P_C, P_D$ in $\DZ$ satisfying 
$P_C + P_D =1$ and 
\eqref{eq:4.1}
\eqref{eq:gaugeZAB}.
\end{proof}
The main purpose of this section is to study the converse implication
of Proposition \ref{prop:SME1}.

We henceforth assume that 
$(\OA,\DA,\rho^A)$ and
$(\OB,\DB,\rho^B)$ are 
strong Morita equivalent in $1$-step
via 
$(\OZ,\DZ,\rho^Z)$
for some matrix $Z$.
We may take two projections $P_A, P_B$ in $\DZ$ 
having the properties $(1), (2), (3) $ and $(4)$ 
in Definition \ref{def:SME}.
Let us denote by $G_Z = (V_Z, E_Z)$ 
the directed graph for the matrix $Z$.
The Cuntz--Krieger algebra 
$\OZ$ is then generated by partial isometries 
$S_\gamma, \gamma \in E_Z$ satisfying the relations:
\begin{equation}
\sum_{\eta\in E_Z} S_\eta S_\eta^* =1, 
\qquad 
S_\gamma^* S_\gamma 
= \sum_{\eta \in E_Z} Z^G(\gamma,\eta) S_\eta S_\eta^* 
\quad \text{ for }
\gamma \in E_Z \label{eq:CKZ}
\end{equation}
where 
$Z^G(\gamma,\eta) =1$ if $t(\gamma) = s(\eta)$, and $0$ otherwise.
We have the following lemmas.
\begin{lemma}\label{lem:SME2}
Let $S_\gamma, \gamma \in E_Z$ be the generating partial isometries 
of $\OZ$ satisfying \eqref{eq:CKZ} . 
Then we have
\begin{enumerate}
\renewcommand{\theenumi}{\roman{enumi}}
\renewcommand{\labelenumi}{\textup{(\theenumi)}}
\item
$P_A S_\gamma P_A = P_B S_\gamma P_B = 0$.
\item
$S_\gamma = P_A S_\gamma P_B + P_B S_\gamma P_A$.
\item
$P_A S_\gamma = S_\gamma P_B$ and $ P_B S_\gamma =  S_\gamma P_A$.
\end{enumerate}
\end{lemma}
\begin{proof}
By the equality $P_A + P_B = 1$, we have 
\begin{equation*}
S_\gamma 
= P_A S_\gamma P_A + P_A S_\gamma P_B + P_B S_\gamma P_A  + P_B S_\gamma P_B.
\end{equation*}
Since $P_A S_\gamma P_A $ belongs to $P_A \OZ P_A $ which is identified with
$\OA$, 
the condition (4) of Definition \ref{def:SME}
gives rise to the equality 
\begin{equation}
\rho^Z_t(P_A S_\gamma P_A ) = \rho^A_{2t}(P_A S_\gamma P_A ). \label{eq:rhoZA}
\end{equation}
As $\rho^Z_t |_{\DZ} = \id $ and $P_A, P_B \in \DZ,$
the left hand side for 
$t = \frac{1}{2}$ of \eqref{eq:rhoZA} goes to
$$
P_A \rho^Z_{\frac{1}{2}}(S_\gamma) P_A = - P_A S_\gamma P_A.
$$
As $\rho^A_{1} = \id$,
the right hand side for
$t = \frac{1}{2}$ goes to
$ P_A S_\gamma P_A$.
Hence we have 
$P_A S_\gamma P_A =0$ and similarly 
$P_B S_\gamma P_B =0.$ 
Therefore we know (i), (ii) and (iii).
\end{proof}
\begin{lemma}\label{lem:SME3}
\begin{equation}
\sum_{\gamma\in E_Z} S_\gamma P_A S_\gamma^* = P_B,
\qquad
\sum_{\gamma\in E_Z} S_\gamma P_B S_\gamma^* = P_A.
\end{equation}
%Hence the projections $P_A$ and $P_B$ are both full projections in $\OZ$.
\end{lemma}
\begin{proof}
By Lemma \ref{lem:SME2}, we know $S_\gamma P_A = P_B S_\gamma$
so that 
\begin{equation}
  \sum_{\gamma\in E_Z} S_\gamma P_A S_\gamma^* 
=\sum_{\gamma\in E_Z} P_B S_\gamma S_\gamma^* 
= P_B.
\end{equation}
Similarly we see that 
$
\sum_{\gamma\in E_Z} S_\gamma P_B S_\gamma^* = P_A.
$
%%%%%%%%%%%%%% 
%As $1_Z = P_A + P_B = P_A + \sum_{\gamma\in E_Z} S_\gamma P_A S_\gamma^*$,
%the unit $1_Z$ of the $C^*$-algebra $\OZ$ is contained in the two-sided ideal
%$\OZ P_A \OZ$ of $\OZ$ generated by $P_A$,
%this shows that the projection $P_A$ is full in $\OZ$, and similarly so is $P_B$.
\end{proof}
 We notice the following identities
which immediately come from Lemma \ref{lem:SME2} (iii).
\begin{lemma}\label{lem:SME4}
For $\gamma_1, \gamma_2 \in E_Z$, we have the following identities.
\begin{enumerate}
\renewcommand{\theenumi}{\roman{enumi}}
\renewcommand{\labelenumi}{\textup{(\theenumi)}}
\item
%\begin{align}
$S_{\gamma_1} S_{\gamma_2} P_A = P_A S_{\gamma_1} S_{\gamma_2} \in \OA
$ 
and
$
S_{\gamma_1} S_{\gamma_2} P_B = P_B S_{\gamma_1} S_{\gamma_2} \in \OB.
$
%\end{align}
\item
%\begin{align}
$
S_{\gamma_1} P_B S_{\gamma_2}  = P_A S_{\gamma_1}P_B S_{\gamma_2}P_A \in \OA
$ 
and
$
S_{\gamma_1} P_A S_{\gamma_2}  = P_B S_{\gamma_1}P_A S_{\gamma_2}P_B \in \OB.
$%\end{align}
\end{enumerate}
\end{lemma}
\begin{lemma}\label{lem:SME5}
Let $\gamma_1, \gamma_2 \in E_Z$.
Then $P_A S_{\gamma_1}\ne 0, P_B S_{\gamma_2}\ne 0$ 
and $Z^G(\gamma_1,\gamma_2) =1$
if and only if
$P_A S_{\gamma_1} S_{\gamma_2}\ne 0$. 
\end{lemma}
\begin{proof}
The if part is obvious. 
It suffices to show the only if part.
Since 
$P_A S_{\gamma_1} S_{\gamma_2}
 =P_A S_{\gamma_1} P_B S_{\gamma_2}
=  S_{\gamma_1} S_{\gamma_2}P_A$,
we have
\begin{align*}
( S_{\gamma_1} S_{\gamma_2}P_A)^*
 S_{\gamma_1} S_{\gamma_2}P_A
& = P_A S_{\gamma_2}^* S_{\gamma_1}^* S_{\gamma_1} S_{\gamma_2}P_A \\
& = \sum_{\eta_1 \in E_Z} Z^G(\gamma_1,\eta_1) 
      P_A S_{\gamma_2}^* S_{\eta_1} S_{\eta_1}^* S_{\gamma_2}P_A \\ 
& =  Z^G(\gamma_1,\gamma_2) 
      P_A S_{\gamma_2}^* S_{\gamma_2}P_A \\
&  = Z^G(\gamma_1,\gamma_2) (P_B S_{\gamma_2})^*(P_B S_{\gamma_2}).
\end{align*}
The above equalities ensure us the only if part.
\end{proof}

\begin{lemma}\label{lem:SME6}
Let $\gamma_1, \gamma_2, \eta_1,\eta_2 \in E_Z$.
Then $S_{\gamma_1}S_{\gamma_2}\ne 0,
S_{\gamma_2}S_{\eta_1}\ne 0, 
 P_A S_{\eta_1} S_{\eta_2}\ne 0$ 
if and only if
$P_A S_{\gamma_1} S_{\gamma_2}S_{\eta_1} S_{\eta_2}\ne 0$. 
\end{lemma}
\begin{proof}
Since
$P_A S_{\gamma_1} S_{\gamma_2}S_{\eta_1} S_{\eta_2}
= S_{\gamma_1} S_{\gamma_2}P_A S_{\eta_1} S_{\eta_2},
$
the if part is obvious. 
It suffices to show the only if part.
We have
\begin{align*}
 & ( P_A S_{\gamma_1} S_{\gamma_2}S_{\eta_1} S_{\eta_2})^*
    ( P_A S_{\gamma_1} S_{\gamma_2}S_{\eta_1} S_{\eta_2}) \\
%=&  S_{\eta_2}^* S_{\eta_1}^* S_{\gamma_2}^* S_{\gamma_1}^* P_A
%     S_{\gamma_1} S_{\gamma_2}S_{\eta_1} S_{\eta_2}\\
=& P_A S_{\eta_2}^* S_{\eta_1}^* S_{\gamma_2}^* S_{\gamma_1}^* 
     S_{\gamma_1} S_{\gamma_2}S_{\eta_1} S_{\eta_2} P_A \\
=& \sum_{\zeta_1 \in E_Z} Z^G(\gamma_1,\zeta_1) 
     P_A S_{\eta_2}^* S_{\eta_1}^* S_{\gamma_2}^* S_{\zeta_1} 
     S_{\zeta_1}^* S_{\gamma_2}S_{\eta_1} S_{\eta_2} P_A \\
%= & Z^G(\gamma_1,\gamma_2) 
 %  P_A S_{\eta_2}^* S_{\eta_1}^* S_{\gamma_2}^* 
 %    S_{\gamma_2}S_{\eta_1} S_{\eta_2} P_A\\
= &  Z^G(\gamma_1,\gamma_2) \sum_{\zeta_2 \in E_Z} Z^G(\gamma_2,\zeta_2) 
  P_A  S_{\eta_2}^* S_{\eta_1}^* S_{\zeta_2} S_{\zeta_2}^* S_{\eta_1} S_{\eta_2}P_A \\
%= & Z^G(\gamma_1,\gamma_2)  Z^G(\gamma_2,\eta_1) 
%  P_A S_{\eta_2}^* S_{\eta_1}^* S_{\eta_1}  S_{\eta_2} P_A \\
= & Z^G(\gamma_1,\gamma_2)  Z^G(\gamma_2,\eta_1) 
     \sum_{\zeta_3 \in E_Z} Z^G(\eta_1,\zeta_3)
  P_A S_{\eta_2}^* S_{\zeta_3} S_{\zeta_3}^*  S_{\eta_2} P_A \\
=&  Z^G(\gamma_1,\gamma_2)  Z^G(\gamma_2,\eta_1)  Z^G(\eta_1,\eta_2)
  P_A S_{\eta_2}^* S_{\eta_2}  P_A .
\end{align*}
The above equalities  ensure us the only if part.
\end{proof}

Now we are assuming that 
the Cuntz--Krieger triplets 
$(\OA,\DA,\rho^A)$ and
$(\OB,\DB,\rho^B)$ 
are strong Morita equivalent in $1$-step
via 
$(\OZ,\DZ,\rho^Z)$.
We introduce several directed graphs in this situation.
Define edge sets $E_{\tilde{A}}, E_{\tilde{B}}, E_{\tilde{C}}, E_{\tilde{D}}$
by setting
\begin{align*}
E_{\tilde{A}} & = \{ (A,\gamma_1\gamma_2) \in \{A\}\times B_2(X_Z) \mid
P_A S_{\gamma_1} S_{\gamma_2} \ne 0 \}, \\
E_{\tilde{B}} & = \{ (B,\gamma_1\gamma_2) \in \{B\}\times B_2(X_Z) \mid
P_B S_{\gamma_1} S_{\gamma_2} \ne 0 \}, \\
E_{\tilde{C}} & = \{ (A,\gamma_1) \in \{A\}\times E_Z \mid
P_A S_{\gamma_1} \ne 0 \}, \\
E_{\tilde{D}} & = \{ (B,\gamma_1) \in \{B\}\times E_Z \mid
P_B S_{\gamma_1} \ne 0 \}
\end{align*} 
and vertex sets 
$V_{\tilde{A}s}, V_{\tilde{A}t},V_{\tilde{B}s}, V_{\tilde{B}t},
V_{\tilde{C}s}, V_{\tilde{C}t},V_{\tilde{D}s}, V_{\tilde{D}t}
$
by setting
\begin{align*}
V_{\tilde{A}s} & = \{ (A, s(\gamma_1)) \in \{A\}\times V_Z \mid
(A, \gamma_1\gamma_2) \in E_{\tilde{A}} \}, \\
V_{\tilde{A}t} & = \{ (A, t(\gamma_2)) \in \{A\}\times V_Z \mid
(A, \gamma_1\gamma_2) \in E_{\tilde{A}} \},\\
V_{\tilde{B}s} & = \{ (B, s(\gamma_1)) \in \{B\}\times V_Z \mid
(B, \gamma_1\gamma_2) \in E_{\tilde{B}} \}, \\
V_{\tilde{B}t} & = \{ (B, t(\gamma_1)) \in \{B\}\times V_Z \mid
(B, \gamma_1\gamma_2) \in E_{\tilde{B}} \}, \\
V_{\tilde{C}s} & = \{ (A, s(\gamma_1)) \in \{A\}\times V_Z \mid
(A, \gamma_1) \in E_{\tilde{C}} \}, \\
V_{\tilde{C}t} & = \{ (B, t(\gamma_1)) \in \{A\}\times V_Z \mid
(A, \gamma_1) \in E_{\tilde{C}} \}, \\
V_{\tilde{D}s} & = \{ (B, s(\gamma_1)) \in \{B\}\times V_Z \mid
(B, \gamma_1) \in E_{\tilde{D}} \}, \\
V_{\tilde{D}t} & = \{ (A, t(\gamma_1)) \in \{B\}\times V_Z \mid
(B, \gamma_1) \in E_{\tilde{D}} \}.
\end{align*} 
\begin{lemma}\label{lem:SME7}
Keep the above notations. We have
\begin{enumerate}
\renewcommand{\theenumi}{\roman{enumi}}
\renewcommand{\labelenumi}{\textup{(\theenumi)}}
\item $V_{\tilde{A}s} = V_{\tilde{A}t}=V_{\tilde{C}s}=V_{\tilde{D}t}.$
\item $V_{\tilde{B}s} = V_{\tilde{B}t}=V_{\tilde{D}s}=V_{\tilde{C}t}.$
\end{enumerate}
\end{lemma}
\begin{proof}
(i) 
We will first show the equality
$V_{\tilde{A}s} = V_{\tilde{A}t}.$
Take an arbitrary vertex $(A, s(\gamma_1) )\in V_{\tilde{A}s}$
and $\gamma_2 \in E_Z$ with
$P_A S_{\gamma_1}S_{\gamma_2} \ne 0$,
so that  
$t(\gamma_1) = s(\gamma_2).$
We may find $\eta_1,\eta_2 \in E_Z$ 
such that 
$S_{\eta_1} S_{\eta_2} \ne 0$ and 
$t(\eta_2)= s(\gamma_1)$.
By Lemma \ref{lem:SME6}, we have
$ S_{\eta_1} S_{\eta_2}S_{\gamma_1} S_{\gamma_2} P_A \ne 0$. 
Since
$ S_{\eta_1} S_{\eta_2}S_{\gamma_1} S_{\gamma_2} P_A 
= P_A S_{\eta_1} S_{\eta_2}S_{\gamma_1} S_{\gamma_2}$,
we have
$P_A S_{\eta_1} S_{\eta_2}\ne 0$
so that 
$(A, t(\eta_2)) \in V_{\tilde{A}t}$
and hence
 $(A, s(\gamma_1) \in V_{\tilde{A}t}$.
This shows that
the inclusion relation
$V_{\tilde{A}s} \subset V_{\tilde{A}t}$
holds.
Similarly we know that 
$V_{\tilde{A}t} \subset V_{\tilde{A}s}$
so that 
$V_{\tilde{A}s}= V_{\tilde{A}t}.$

We will second show the equality
$V_{\tilde{C}s} = V_{\tilde{D}t}.$
Take an arbitrary vertex $(A, s(\gamma_1))\in V_{\tilde{C}s}.$
We see that 
$P_A S_{\gamma_1} \ne 0$
and hence
$S_{\gamma_1}P_B \ne 0$.
As 
$\sum_{\gamma' \in E_Z} S_{\gamma'}^*S_{\gamma'}\ge 1$,
We may find $\gamma_2 \in E_Z$
such that
$S_{\gamma_2}S_{\gamma_1}P_B \ne 0$
so that 
$t(\gamma_2) = s(\gamma_1)$.
Since 
$S_{\gamma_2}S_{\gamma_1}P_B = 
P_BS_{\gamma_2} S_{\gamma_1},$ 
we have
$P_BS_{\gamma_2} \ne 0$.
This implies that
$(B, \gamma_2) \in E_{\tilde{D}}$
and
$(A, t(\gamma_2)) \in V_{\tilde{D}t}$.
As $t(\gamma_2) = s(\gamma_1)$,
we obtain that 
$(A, s(\gamma_1)) \in V_{\tilde{D}t}$
so that 
$V_{\tilde{C}s} \subset V_{\tilde{D}t}$.
We similarly see that 
$V_{\tilde{D}t} \subset V_{\tilde{C}s}$
so that 
$V_{\tilde{C}s} = V_{\tilde{D}t}$.

We will finally show that  
$V_{\tilde{A}s} = V_{\tilde{C}s}$.
Since the condition
$P_A S_{\gamma_1}S_{\gamma_2} \ne 0$
implies
$P_A S_{\gamma_1} \ne 0,$
we have
$V_{\tilde{A}s} \subset V_{\tilde{C}s}$.
Conversely, 
for $(A, s(\gamma_1)) \in V_{\tilde{C}s}$,
we have
$P_A S_{\gamma_1} \ne 0$
so that 
$S_{\gamma_1}P_B \ne 0$.
Since
$P_B = \sum_{\gamma' \in E_Z}S_{\gamma'}P_A S_{\gamma'}^*$,
we may find 
$\gamma_2 \in E_Z$ such that 
$S_{\gamma_1} S_{\gamma_2} P_A \ne 0$.
Hence we see that 
$P_A S_{\gamma_1} S_{\gamma_2}\ne 0$
so that 
$(A,s(\gamma_1)) \in V_{\tilde{A}s}$.
This shows that 
$V_{\tilde{A}s} = V_{\tilde{C}s}$.
Therefore (i) has been shown.
(ii) is similarly shown.
\end{proof}
Let us denote by 
$V_{\tilde{A}}$ and 
by $V_{\tilde{B}}$ 
the first four vertex sets and 
the second four vertex sets   
in Lemma \ref{lem:SME7},
respectively. 
Namely we put 
\begin{align*}
V_{\tilde{A}} &: = V_{\tilde{A}s} = V_{\tilde{A}t}=V_{\tilde{C}s}=V_{\tilde{D}t},\\
V_{\tilde{B}} &: = V_{\tilde{B}s} = V_{\tilde{B}t}=V_{\tilde{D}s}=V_{\tilde{C}t}.
\end{align*}
For an edge $(A, \gamma_1 \gamma_2) \in E_{\tilde{A}}$,
define its source and terminal vertices by
$$
s(A,\gamma_1 \gamma_2) =(A,s(\gamma_1)) \in V_{\tilde{A}s}, \qquad
t(A,\gamma_1\gamma_2) =(A,t(\gamma_2)) \in V_{\tilde{A}t}.
$$
We then have a directed graph
$(V_{\tilde{A}}, E_{\tilde{A}})$ denoted by
$
G_{\tilde{A}}.
$
We similarly have a directed graph
$ 
G_{\tilde{B}} =(V_{\tilde{B}}, E_{\tilde{B}}).
$ 
From an  edge 
 $(A, \gamma_1) \in E_{\tilde{C}}$,
 define its source and terminal 
vertices by
$$
s(A,\gamma_1) =(A,s(\gamma_1)) \in V_{\tilde{C}s}, \qquad
t(A,\gamma_1) =(A,t(\gamma_1)) \in V_{\tilde{C}t}.
$$
We have a directed graph 
$G_{\tilde{C}} =(V_{\tilde{A}}\overset{E_{\tilde{C}}}{\longrightarrow}V_{\tilde{B}})$
and similarly
$G_{\tilde{D}} =(V_{\tilde{B}}\overset{E_{\tilde{D}}}{\longrightarrow}V_{\tilde{A}}).$

Let 
$\tilde{A}$  be the vertex transition matrix 
$\tilde{A}: V_{\tilde{A}}\times V_{\tilde{A}}\longrightarrow \Zp
$
 of the directed graph 
$
G_{\tilde{A}}
$
which is defined by 
\begin{equation*}
\tilde{A}((A,u), (A,v))
= | \{ (A,\gamma_1 \gamma_2) \in E_{\tilde{A}} 
\mid s(\gamma_1) = u, t(\gamma_2) = v\} | 
\end{equation*}
for $(A,u), (A,v) \in V_{\tilde{A}}.$
The edge transition matrix 
$\tilde{A}^G: E_{\tilde{A}}\times E_{\tilde{A}}\longrightarrow \{0,1\}
$
 of  
$
G_{\tilde{A}}
$
 is defined by 
\begin{equation*}
\tilde{A}^G(\gamma_1 \gamma_2,\eta_1 \eta_2)
=
\begin{cases}
1 & \text{ if } t(A,\gamma_1\gamma_2)=s(A,\eta_1\eta_2),\\
0 & \text{ otherwise}
\end{cases}  
\end{equation*}
for $(A,\gamma_1\gamma_2), (A,\eta_1\eta_2) \in E_{\tilde{A}}.$
We similarly have the vertex transition matrices 
$\tilde{B}, \tilde{C}, \tilde{D}$
and 
the edge  transition matrices 
$\tilde{B}^G, \tilde{C}^G, \tilde{D}^G$
of the directed graphs 
$
G_{\tilde{B}}, G_{\tilde{C}}, G_{\tilde{D}}, 
$
respectively.
\begin{proposition} \label{prop:SME8}
The matrices $\tilde{A}$ and $\tilde{B}$ are elementary equivalent such that
\begin{equation*}
\tilde{A}=\tilde{C}\tilde{D} \quad \text{ and }
\quad
\tilde{B}= \tilde{D}\tilde{C}.
\end{equation*}
Hence 
$\tilde{A}^G=\tilde{C}^G\tilde{D}^G
$
and
$
\tilde{B}^G= \tilde{D}^G\tilde{C}^G,
$
and
the two-sided topological Markov shifts
$(\bar{X}_{\tilde{A}}, \bar{\sigma}_{\tilde{A}})$
and
$(\bar{X}_{\tilde{B}}, \bar{\sigma}_{\tilde{B}})$
are topologically conjugate.
\end{proposition}
\begin{proof}
For  $(A,\gamma_1 \gamma_2)$ with $\gamma_1, \gamma_2 \in E_{Z}$,
Lemma \ref{lem:SME5}
ensures us that
$(A, \gamma_1) \in E_{\tilde{C}}, (B, \gamma_2) \in E_{\tilde{D}}, 
Z^G(\gamma_1,\gamma_2)=1$
if and only if
$(A,\gamma_1 \gamma_2) \in E_{\tilde{A}}.$
Since 
$t(A,\gamma_1) = s(B,\gamma_2)$ 
if and only if 
$Z^G(\gamma_1,\gamma_2)=1$,
we know that 
$\tilde{A}=\tilde{C}\tilde{D}$,
and similarly 
$\tilde{B}= \tilde{D}\tilde{C}.$
The relations
$\tilde{A}^G=\tilde{C}^G\tilde{D}^G
$
and
$
\tilde{B}^G= \tilde{D}^G\tilde{C}^G
$
automatically come from 
$
\tilde{A}=\tilde{C}\tilde{D}
$ 
and
$
\tilde{B}= \tilde{D}\tilde{C}.
$
\end{proof}
Let
$E_{\tilde{Z}} = E_{\tilde{C}} \cup E_{\tilde{D}}$
and
$V_{\tilde{Z}} = V_{\tilde{A}} \cup E_{\tilde{B}}.$
We have a bipartite directed graph
$G_{\tilde{Z}} =(V_{\tilde{Z}}, E_{\tilde{Z}})$.
Let us denote by 
$\tilde{Z}$ and $\tilde{Z}^G$ the vertex transition matrix
and the edge transition matrix  of the directed graph
$G_{\tilde{Z}}$, respectively.
Since $G_{\tilde{Z}}$ is bipartite, by the above proposition,
we have
\begin{equation*}
\tilde{Z} =
\begin{bmatrix}
0 & \tilde{C} \\
\tilde{D} & 0
\end{bmatrix},
\qquad 
\tilde{Z}^2 
=
\begin{bmatrix}
\tilde{A} & 0 \\
0 & \tilde{B}
\end{bmatrix}.
\end{equation*}

We will study the relationship between 
the two matrices
$\tilde{Z}$ and $Z$. 
For $\gamma \in E_Z$, 
denote by
$S_{(A,\gamma)}, S_{(B,\gamma)}$
the partial isometries
$P_A S_\gamma, P_B S_\gamma,$
respectively,
so that 
$S_\gamma = S_{(A,\gamma)}+ S_{(B,\gamma)}$.
\begin{lemma}\label{lem:SME9}
Let $\gamma_1, \gamma_2 \in E_Z$ satisfy 
$Z^G(\gamma_1, \gamma_2)=1$.
\begin{enumerate}
\renewcommand{\theenumi}{\roman{enumi}}
\renewcommand{\labelenumi}{\textup{(\theenumi)}}
\item $S_{(B,\gamma_2)} \ne 0$ implies $S_{(A,\gamma_1)} \ne 0.$
\item $S_{(A,\gamma_2)} \ne 0$ implies $S_{(B,\gamma_1)} \ne 0.$
\end{enumerate}
\end{lemma}
\begin{proof}
(i)
Since 
$S_{(A,\gamma_1)} S_{(B,\gamma_2)}
 =P_A S_{\gamma_1} P_B S_{\gamma_2}
=  S_{\gamma_1} S_{\gamma_2}P_A$,
we have
\begin{align*}
( S_{(A,\gamma_1)} S_{(B,\gamma_2)})^*
( S_{(A,\gamma_1)} S_{(B,\gamma_2)})
& = P_A S_{\gamma_2}^* S_{\gamma_1}^* S_{\gamma_1} S_{\gamma_2}P_A \\
& = \sum_{\eta_1 \in E_Z} Z^G(\gamma_1,\eta_1) 
      P_A S_{\gamma_2}^* S_{\eta_1} S_{\eta_1}^* S_{\gamma_2}P_A \\ 
%& =  Z^G(\gamma_1,\gamma_2) 
 %     P_A S_{\gamma_2}^* S_{\gamma_2}P_A \\
& =  Z^G(\gamma_1,\gamma_2)  S_{(B,\gamma_2)}^* S_{(B,\gamma_2)}.
\end{align*}
The above equality ensures us the assertion.
(ii) is similarly shown.
\end{proof}
\begin{lemma}
Either of the following two situations occurs:
\begin{enumerate}
\item
Both $S_{(A,\gamma)}$ and $S_{(B,\gamma)}$
are not zero for all $\gamma\in E_Z$.
In this case we have $\tilde{C}^G =\tilde{D}^G = Z^G$
so that $\tilde{A} = \tilde{B}$ and
$\tilde{Z} =
\begin{bmatrix}
0 & Z \\
Z & 0
\end{bmatrix}
$
\item 
Either $S_{(A,\gamma)}=0$ or  $S_{(B,\gamma)}=0$
 for all $\gamma\in E_Z$.
In this case we have  
$\tilde{Z}=Z.$ 
\end{enumerate}
\end{lemma}
\begin{proof}
Suppose that there exists $\gamma_0 \in E_Z$ 
such that both conditions 
$S_{(A,\gamma_0)} \ne 0$ and $S_{(B,\gamma_0)} \ne 0$
hold.
By the preceding lemma, any edge $\eta \in E_Z$ satisfying
$Z^G(\eta,\gamma_0) =1$ forces that
$S_{(A,\eta)} \ne 0$ and $S_{(B,\eta)} \ne 0.$
Since for any edge $\gamma \in E_Z$, there exists a finite sequence
of edges $\gamma_1, \dots, \gamma_n$ in $E_Z$ such that 
$$
  Z^G(\eta,\gamma_1)
=Z^G(\gamma_1,\gamma_2)
=\cdots 
=Z^G(\gamma_n,\gamma_0)=1
$$
so that  
 $S_{(A,\gamma)} \ne 0$ and $S_{(B,\gamma)} \ne 0.$
Hence 
either of the following two cases occurs:
\begin{enumerate}
\item
Both $S_{(A,\gamma)}$ and $S_{(B,\gamma)}$
are not zero for all $\gamma\in E_Z$.
\item 
Either $S_{(A,\gamma)}=0$ or  $S_{(B,\gamma)}=0$
 for all $\gamma\in E_Z$.
\end{enumerate}
Case (1):
We have the following equalities.
\begin{align*}
S_\gamma^* S_\gamma
& = (S_{(A,\gamma)}^* + S_{(B,\gamma)}^*)(S_{(A,\gamma)} + S_{(B,\gamma)}) \\
& = S_{(A,\gamma)}^*S_{(A,\gamma)} + S_{(B,\gamma)}^* S_{(B,\gamma)} \\
& = \sum_{(B,\eta)\in E_{\tilde{D}}}\tilde{C}^G((A,\gamma),(B,\eta))
    S_{(B,\eta)} S_{(B,\eta)}^* + 
    \sum_{(A,\eta)\in E_{\tilde{C}}}\tilde{D}^G((B,\gamma),(A,\eta))
    S_{(A,\eta)} S_{(A,\eta)}^*. 
\end{align*}
On the other hand, we have
\begin{align*}
S_\gamma^* S_\gamma
& = \sum_{\eta \in E_Z} Z^G(\gamma,\eta) S_{\eta}S_{\eta}^* \\
& = \sum_{\eta \in E_Z} Z^G(\gamma,\eta)
     (P_B S_{\eta}S_{\eta}^* P_B + P_A S_{\eta}S_{\eta}^* P_A)  \\
& = \sum_{\eta \in E_Z} Z^G(\gamma,\eta)
    S_{(B,\eta)} S_{(B,\eta)}^* + 
    \sum_{\eta \in E_Z} Z^G(\gamma,\eta)
    S_{(A,\eta)} S_{(A,\eta)}^*. 
\end{align*}
Since both $S_{(A,\gamma)}\ne 0$ and $S_{(B,\gamma)}\ne 0$
 for all $\gamma\in E_Z$,
we have 
\begin{equation*}
\tilde{C}^G((A,\gamma),(B,\eta)) = Z^G(\gamma,\eta), \qquad
\tilde{D}^G((B,\gamma),(A,\eta)) = Z^G(\gamma,\eta)
\end{equation*}
for all $\gamma, \eta \in E_Z$.
Hence we have
$\tilde{C}^G =\tilde{D}^G = Z^G$
so that $\tilde{A}^G = \tilde{B}^G$
and hence $\tilde{A} = \tilde{B}$.
As
$\tilde{Z} =
\begin{bmatrix}
0 & \tilde{C} \\
\tilde{D} & 0
\end{bmatrix},
$
we have
$
\tilde{Z}^G =
\begin{bmatrix}
0 & Z^G \\
Z^G & 0
\end{bmatrix}
$
and hence
$\tilde{Z} =
\begin{bmatrix}
0 & Z \\
Z & 0
\end{bmatrix}.
$

Case (2): 
Since either $S_{(A,\gamma)}\ne 0$ or  $S_{(B,\gamma)}\ne0$
 for all $\gamma\in E_Z$
occurs, we have a disjoint union
$E_Z = E_{\tilde{C}} \cup E_{\tilde{D}}.$
As 
$S_{(A,\gamma_1)}S_{(A,\gamma_2)} =0,
 S_{(B,\gamma_1)}S_{(B,\gamma_2)} =0
$
for all $\gamma_1,\gamma_2 \in E_Z$,
we have
$Z =
\begin{bmatrix}
0 & \tilde{C} \\
\tilde{D} & 0
\end{bmatrix}
$
so that 
$\tilde{Z} = Z$.
%In this case, for 
%$S_{\gamma_i} = S_{(A,\gamma_i)} + S_{(B,\gamma_i)} \in E_Z, i=1,2,$ 
%we have
%\begin{equation*}
%S_{\gamma_1}S_{\gamma_2} =
%S_{(A,\gamma_1)}S_{(B,\gamma_2)} +S_{(B,\gamma_1)}S_{(A,\gamma_2)}.
%\end{equation*}
%This argument shows the desired assertion.
\end{proof}

\medskip

We will next study the bipartite graph $G_{\tilde{Z}}$ from the $C^*$-algebraic view point.
For $(A,\gamma_1 \gamma_2) \in E_{\tilde{A}}$, define the partial isometry
$$
S_{(A,\gamma_1 \gamma_2)} = P_A S_{\gamma_1} S_{\gamma_2}.
$$
\begin{lemma}\label{lem:SME11}
The $C^*$-subalgebra 
$C^*(S_{(A,\gamma_1 \gamma_2)};(A,\gamma_1 \gamma_2) \in E_{\tilde{A}})$
of $\OZ$ is isomorphic to the Cuntz--Krieger algebra 
${\mathcal{O}}_{\tilde{A}}$ for the matrix $\tilde{A}$. 
\end{lemma}
\begin{proof}
We first notice that
\begin{equation*}
\sum_{(A,\gamma_1 \gamma_2) \in E_{\tilde{A}}}
     S_{(A,\gamma_1 \gamma_2)}S_{(A,\gamma_1 \gamma_2)}^*
=
\sum_{\gamma_1, \gamma_2 \in E_Z}
     P_A S_{\gamma_1}S_{\gamma_2} S_{\gamma_2}^* S_{\gamma_1}^* P_A
= P_A.
\end{equation*}
We also have
\begin{align*}
S_{(A,\gamma_1 \gamma_2)}^*S_{(A,\gamma_1 \gamma_2)}
& = P_A S_{\gamma_2}^* S_{\gamma_1}^* S_{\gamma_1} S_{\gamma_2}P_A \\
& = \sum_{\zeta_1 \in E_Z} Z^G(\gamma_1,\zeta_1) 
      P_A S_{\gamma_2}^* S_{\zeta_1} S_{\zeta_1}^* S_{\gamma_2}P_A \\ 
%& =  Z^G(\gamma_1,\gamma_2) 
%     P_A S_{\gamma_2}^* S_{\gamma_2}P_A \\
&  = \sum_{\eta_1 \in E_Z} 
     Z^G(\gamma_1,\gamma_2) Z^G(\gamma_2,\eta_1)
     P_A S_{\eta_1} S_{\eta_1}^*P_A \\
%&  = Z^G(\gamma_1,\gamma_2) 
%   \sum_{\eta_1, \eta_2 \in E_Z} Z^G(\gamma_2,\eta_1)Z^G(\eta_1,\eta_2)
%   P_A S_{\eta_1}S_{\eta_2}S_{\eta_2}^* S_{\eta_1}^*P_A  \\
&  = \sum_{\eta_1, \eta_2 \in E_Z}  
      Z^G(\gamma_1,\gamma_2) Z^G(\gamma_2,\eta_1)Z^G(\eta_1,\eta_2)
     P_A S_{\eta_1}S_{\eta_2}S_{\eta_2}^* S_{\eta_1}^*P_A 
\end{align*}
For $(A,\gamma_1 \gamma_2), (A,\eta_1 \eta_2) \in E_{\tilde{A}}$,
the condition
$t(A,\gamma_1 \gamma_2) =s(A,\eta_1 \eta_2)$
holds if and only if
$Z^G(\gamma_2,\eta_1) =1$.
Hence we know 
$$ 
Z^G(\gamma_1,\gamma_2) Z^G(\gamma_2,\eta_1)Z^G(\eta_1,\eta_2)
= {\tilde{A}}^G(\gamma_1\gamma_2, \eta_1 \eta_2).
$$
By the above equalities, 
we have
\begin{equation*}
S_{(A,\gamma_1 \gamma_2)}^*S_{(A,\gamma_1 \gamma_2)}
  = \sum_{(A,\eta_1 \eta_2) \in E_{\tilde{A}}} 
     {\tilde{A}}^G(\gamma_1\gamma_2, \eta_1 \eta_2) 
     S_{(A,\eta_1 \eta_2)}S_{(A,\eta_1 \eta_2)}^*. 
\end{equation*}
Hence the $C^*$-subalgebra 
$C^*(S_{(A,\gamma_1 \gamma_2)};(A,\gamma_1 \gamma_2) \in E_{\tilde{A}})$
of $\OZ$ is isomorphic to the Cuntz--Krieger algebra 
${\mathcal{O}}_{\tilde{A}}$ for the matrix $\tilde{A}$.
\end{proof}
\begin{lemma}\label{lem:SME12}
The $C^*$-subalgebra 
$C^*(S_{(A,\gamma_1 \gamma_2)};(A,\gamma_1 \gamma_2) \in E_{\tilde{A}})$
of $\OZ$ is nothing but $P_A \OZ P_A$.
Hence  the Cuntz--Krieger algebra 
${\mathcal{O}}_{\tilde{A}}$ is isomorphic to $\OA$. 
\end{lemma}
\begin{proof}
Since
$
S_{(A,\gamma_1 \gamma_2)} = P_A S_{\gamma_1} S_{\gamma_2}P_A
$
for
$(A,\gamma_1 \gamma_2) \in E_{\tilde{A}}$,
we have
$C^*(S_{(A,\gamma_1 \gamma_2)};(A,\gamma_1 \gamma_2) \in E_{\tilde{A}})
\subset P_A \OZ P_A$.
We will show the converse inclusion relation.
Take an arbitrary fixed $X \in \OZ$ with $P_A X P_A \ne 0$.
Let ${\mathcal{P}}_Z$ be the dense $*$-subalgebra of $\OZ$ algebraically generated by
$S_\gamma, \gamma \in E_Z$.
We may find $X_n \in {\mathcal{P}}_Z$ such that 
$\| X - X_n \| \rightarrow 0$.
Since 
$\| P_A X P_A - P_A X_n P_A \| \le \| X - X_n \| \rightarrow 0,$
it suffices to show that 
$P_A X_n P_A$ belongs to 
$C^*(S_{(A,\gamma_1 \gamma_2)};(A,\gamma_1 \gamma_2) \in E_{\tilde{A}})$.
By \cite[2.2 Lemma]{CK},
any element of the subalgebra ${\mathcal{P}}_Z$
is a finite linear combination of elements of the form
$S_\mu S_i S_i^* S_\nu^*$
for some $\mu =(\mu_1,\dots,\mu_m),
              \nu =(\nu_1,\dots,\nu_n) \in B_*(X_Z).$
Assume that 
$P_A S_\mu S_i S_i^* S_\nu^* P_A \ne 0.$
Since
$P_A S_j = S_j P_B$, we have
\begin{equation}
P_A S_\mu 
= P_A S_{\mu_1}\cdots S_{\mu_m}
= 
\begin{cases}
S_{\mu_1}\cdots S_{\mu_m} P_A  & \text{ if } m \text{ is even, }\\
S_{\mu_1}\cdots S_{\mu_m} P_B  & \text{ if } m \text{ is odd. }
\end{cases}
\end{equation}
The assumption  
$P_A S_\mu S_i S_i^* S_\nu^* P_A \ne 0$ 
forces the numbers $m, n$ to be both even, or both odd.

Case 1: $m, n$ are both even.

We have 
\begin{align*}
& P_A S_\mu S_i S_i^* S_\nu^* P_A \\
= & P_A S_{\mu_1}S_{\mu_2}P_A S_{\mu_3}S_{\mu_4} P_A \cdots 
      P_A S_{\mu_{m-1}}S_{\mu_m} P_A S_i S_i^* P_A
     S_{\nu_n}^* S_{\nu_{n-1}}^* P_A \cdots 
    S_{\nu_4}^* S_{\nu_3}^* P_A S_{\nu_2}^* S_{\nu_1}^* P_A \\
= & S_{(A,\mu_1\mu_2)} S_{(A,\mu_3\mu_4)} \cdots 
      S_{(A,\mu_{m-1}\mu_m)} P_A S_i S_i^* P_A 
     S_{(A, \nu_{n-1}\nu_n)}^*  \cdots 
    S_{(A,\nu_3 \nu_4)}^* S_{(A, \nu_1 \nu_2)}^*.
\end{align*}
Now we have
\begin{equation*}
P_A S_i S_i^* P_A  
= \sum_{j \in E_Z} P_A S_i S_j S_j^* S_i^* P_A
= \sum_{j \in E_Z} S_{(A,ij)}S_{(A,ij)}^*
\end{equation*}
so that 
$P_A S_\mu S_i S_i^* S_\nu^* P_A$
is a finite linear combination of products of the elements
$
S_{(A,\gamma_1 \gamma_2)},
S_{(A,\gamma_1 \gamma_2)}^*
$
for
$(A,\gamma_1 \gamma_2) \in E_{\tilde{A}}$
and hence 
it belongs to
$C^*(S_{(A,\gamma_1 \gamma_2)};(A,\gamma_1 \gamma_2) \in E_{\tilde{A}})$.

Case 2: $m, n$ are both odd.

Similarly to Case 1, we have 
\begin{equation*}
 P_A S_\mu S_i S_i^* S_\nu^* P_A 
= S_{(A,\mu_1\mu_2)} \cdots 
      S_{(A,\mu_{m-2}\mu_{m-1})}
      S_{(A,\mu_m i)}S_{(A,\nu_n i)}^* 
     S_{(A, \nu_{n-2}\nu_{n-1})}^*  \cdots 
      S_{(A, \nu_1 \nu_2)}^*
\end{equation*}
so that 
$P_A S_\mu S_i S_i^* S_\nu^* P_A$
 belongs to
$C^*(S_{(A,\gamma_1 \gamma_2)};(A,\gamma_1 \gamma_2) \in E_{\tilde{A}})$.
\end{proof}

\begin{proposition}\label{prop:SME13}
The Cuntz--Krieger triplet 
$({\mathcal{O}}_{\tilde{A}}, {\mathcal{D}}_{\tilde{A}}, {\rho}^{\tilde{A}})$
for the matrix $\tilde{A}$
is isomorphic to
$(\OA, \DA, {\rho}^{A})$.  
\end{proposition}
\begin{proof}
By Lemma \ref{lem:SME11} and Lemma \ref{lem:SME12},
we know that 
\begin{equation}
{\mathcal{O}}_{\tilde{A}}
=C^*(S_{(A,\gamma_1 \gamma_2)};(A,\gamma_1 \gamma_2) \in E_{\tilde{A}})
= P_A {\mathcal{O}}_Z P_A =\OA.
\end{equation}
Under the identification
between 
$C^*(S_{(A,\gamma_1 \gamma_2)};(A,\gamma_1 \gamma_2) \in E_{\tilde{A}})$
and
$P_A {\mathcal{O}}_Z P_A$
in Lemma \ref{lem:SME12},
 the $C^*$-subalgebra
$$
C^*(S_{(A,\gamma_1 \gamma_2)}\cdots S_{(A,\gamma_{n-1} \gamma_n)}
S_{(A,\gamma_{n-1} \gamma_n)}^* \cdots S_{(A,\gamma_1 \gamma_2)}^*;
(A,\gamma_1 \gamma_2), \dots, (A,\gamma_{n-1} \gamma_n) \in E_{\tilde{A}})
$$
of
$C^*(S_{(A,\gamma_1 \gamma_2)};(A,\gamma_1 \gamma_2) \in E_{\tilde{A}})$
generated by the projections
$$
S_{(A,\gamma_1 \gamma_2)}\cdots S_{(A,\gamma_{n-1} \gamma_n)}
S_{(A,\gamma_{n-1} \gamma_n)}^* \cdots S_{(A,\gamma_1 \gamma_2)}^*
$$
for 
$
(A,\gamma_1 \gamma_2), \dots, (A,\gamma_{n-1} \gamma_n) \in E_{\tilde{A}}
$
is naturally identified with
the $C^*$-subalgebra
$P_A \DZ P_A$
of $\DZ$.
Hence we know that
${\mathcal{D}}_{\tilde{A}} = \DA$.
By regarding the generating partial isometry
$S_{(A,\gamma_1 \gamma_2)}$ 
for 
$(A,\gamma_1 \gamma_2)\in E_{\tilde{A}}$
  as an element 
of $P_A \OZ P_A = \OA$, we have
\begin{align*}
\rho^{\tilde{A}}_{2t}(S_{(A,\gamma_1 \gamma_2)})
= & e^{2\pi\sqrt{-1} 2t}S_{(A,\gamma_1 \gamma_2)} \\
= & P_A e^{2\pi\sqrt{-1} t}S_{\gamma_1}e^{2\pi\sqrt{-1} t}S_{\gamma_2} \\
= & P_A \rho^Z_t(S_{\gamma_1}) \rho^Z_t(S_{\gamma_2}) \\
= & \rho^Z_t(P_A S_{\gamma_1}S_{\gamma_2}) 
\end{align*}
Since
$P_A S_{\gamma_1}S_{\gamma_2} \in P_A \OZ P_A=\OA$
and
$\rho^Z_t|_{P_A \OZ P_A} = \rho^A_{2t}$ on $\OA$,
we have
\begin{equation*}
 \rho^Z_t(P_A S_{\gamma_1}S_{\gamma_2})
=\rho^A_{2t}( P_A S_{\gamma_1}S_{\gamma_2})
=\rho^A_{2t}( S_{(A,\gamma_1 \gamma_2)})
\end{equation*}
so that  
$\rho^{\tilde{A}}_{2t} =\rho^A_{2t}$ for all $t \in \T$
and hence
$\rho^{\tilde{A}} =\rho^A$.
\end{proof}
We thus have
\begin{proposition}\label{prop:SMEmain}
Suppose that the Cuntz--Krieger triplets
$(\OA,\DA,\rho^A)$ and
$(\OB,\DB,\rho^B)$ 
are strong Morita equivalent in $1$-step.
Then 
the two-sided topological Markov shifts
$(\bar{X}_A, \bar{\sigma}_A)$
and
$(\bar{X}_B, \bar{\sigma}_B)$ are topologically conjugate.
\end{proposition}
\begin{proof}
Assume that
 the Cuntz--Krieger triplets
$(\OA,\DA,\rho^A)$ and
$(\OB,\DB,\rho^B)$ 
are strong Morita equivalent in $1$-step.
By Proposition \ref{prop:SME8},
the matrices $\tilde{A}, \tilde{B}$ 
are elementary equivalent so that 
their two-sided topological Markov shifts 
$(\bar{X}_{\tilde{A}}, \bar{\sigma}_{\tilde{A}})$
and
$(\bar{X}_{\tilde{B}}, \bar{\sigma}_{\tilde{B}})$ 
are topologically conjugate.
Proposition \ref{prop:SME13} with \cite[Corollary 3.5]{MaMZ2016}
ensures us that
the ons-sided topological Markov shifts
$({X}_{\tilde{A}}, {\sigma}_{\tilde{A}})$
and
$(X_A, \sigma_A)$ 
are eventually conjugate and hence strongly continuous orbit equivalent
in the sense of \cite{MaMZ2016}.
Since the latter property
yields topological conjugacy of their two-sided topological Markov shifts,
the two-sided topological Markov shifts 
$(\bar{X}_{\tilde{A}}, \bar{\sigma}_{\tilde{A}})$
and
$(\bar{X}_{A}, \bar{\sigma}_{A})$ 
are topologically conjugate.
Similarly we know that 
the two-sided topological Markov shifts 
$(\bar{X}_{\tilde{B}}, \bar{\sigma}_{\tilde{B}})$
and
$(\bar{X}_{A}, \bar{\sigma}_{B})$ 
are topologically conjugate.
Therefore we get the assertion.
\end{proof}

Now we reach one of the main results of the paper.
\begin{theorem}
Let $A, B$ be irreducible non-permutation matrices.
The Cuntz--Krieger triplets
$(\OA,\DA,\rho^A)$ and
$(\OB,\DB,\rho^B)$ 
are strong Morita equivalent 
if and only if
their two-sided topological Markov shifts
$(\bar{X}_A, \bar{\sigma}_A)$
and
$(\bar{X}_B, \bar{\sigma}_B)$ are topologically conjugate.
\end{theorem}
\begin{proof}
If part comes from Proposition \ref{prop:SME1}.
The only if part follows from Proposition \ref{prop:SMEmain}. 
\end{proof}

As a corollary we have
\begin{corollary}\label{cor:SMECK}
Let $A, B$ be irreducible non-permutation matrices.
The Cuntz--Krieger triplets
$(\OA,\DA,\rho^A)$ and
$(\OB,\DB,\rho^B)$ 
are strong Morita equivalent 
if and only if the matrices
$A$ and $B$ are strong shift equivalent.
\end{corollary}

%%%%%%%%%%%%%%%%%%%%%%%%%%%%%%%%%%%%%%%%%%%%%%%%%%%%%%%%
%%%%%%%%%%%%%%%%%%%%%%%%%%%%%%%%%%%%%%%%%%%%%%%%%%%%%%%%
\section{Strong shift equivalence and circle actions on $\OA$}
%%%%%%%%%%%%%%%%%%%%%%%%%%%%%%%%%%%%%%%%%%%%%%%%%%%%%%%%
%%%%%%%%%%%%%%%%%%%%%%%%%%%%%%%%

It is well-known that two unital $C^*$-algebras 
$\A$ and $\B$
are strong Morita equivalent 
if and only if their stabilizations 
$\A\otimes\K$ and $\B\otimes\K$ 
are isomorphic  by Brown--Green--Rieffel Theorem \cite[Theorem 1.2]{BGR} 
(cf. \cite{BGR}, \cite{Combes}).
We will next study relationships between stabilized Cuntz--Krieger algebras
with their gauge actions
and strong shift equivalence matrices.
 We will investigate stabilizations of generalized gauge actions from a view point of 
flow equivalence. 

Recall that 
for a function $f\in C(X_A,\Z)$ 
and $t \in \T$,
an automorphism
$\rho_t^{A,f}\in\Aut(\OA)$ 
is defined by 
$
\rho_t^{A,f}(S_i) = U_t(f) S_i, i=1,\dots,N, t \in \T
$
for the unitary
$
U_t(f) = \exp({2\pi \sqrt{-1} t f})\in \DA
$
as in \eqref{eq:rhotf}.
It is easy to see that
the automorphisms
$\rho_t^{A,f}, t \in \T$ 
yield an action of 
$\T$ to $\OA$
 such that
$\rho^{A,f}_t(a) =a$
for all $a \in \DA$.
For $f \in C(X_A,\Z)$ 
and $n \in \Zp$, let us denote by $f^n$ 
the function 
$f^n(x) =\sum_{i=0}^{n-1}f (\sigma_A^i(x)), x \in X_A$.
We know that the following identity holds (cf. \cite[Lemma 3.1]{MaMZ2016})
\begin{equation}
\rho_t^{A,f}(S_\mu) = U_t(f^n)S_\mu,
\qquad f \in C(X_A,\Z),\,
\mu = (\mu_1,\dots,\mu_n) \in B_n(X_A), \,
 t \in \T.
\end{equation}

\medskip

For a $C^*$-algebra $\mathcal{A}$ without unit,
let $M({\mathcal{A}})$ stand for its multiplier $C^*$-algebra
defined by 
\begin{equation*}
M({\mathcal{A}}) = \{ a \in {\mathcal{A}}^{**} 
\mid a {\mathcal{A}} \subset {\mathcal{A}}, \, 
    {\mathcal{A}} a \subset {\mathcal{A}} \}
\end{equation*}
where ${\mathcal{A}}^{**}$ denotes the second dual
${({\mathcal{A}}^{*})}^{*}$ of the $C^*$-algebra $\mathcal{A}$.
An action $\alpha$ of $\T$ to $\mathcal{A}$ extends to 
  $M({\mathcal{A}})$ and is still denoted by $\alpha$.
For an action $\alpha$ of $\T$ to ${\mathcal{A}}$,
a unitary one-cocycle $u_t, t \in \T$ relative to $\alpha$
is a continuous map $t \in \T \rightarrow u_t \in \U(M({\mathcal{A}}))$
to the unitary group  $\U(M({\mathcal{A}}))$ 
satisfying $u_{t+s} = u_s \alpha_s(u_t), s,t \in \T$.
The following proposition has been proved in \cite{MaMZ2016}.
\begin{proposition}[{\cite[Proposition 4.3]{MaMZ2016}}] \label{prop:4.3}
Suppose that 
$A = CD$ and $B =DC$.
Then there exists an isomorphism
$\Phi:\SOA \rightarrow \SOB$ satisfying 
$\Phi(\SDA) = \SDB$
and a homomorphism $\phi:C(X_A,\Z) \rightarrow C(X_B,\Z)$ of ordered groups
%which induces an isomorphism between $H^A$ and $H^B$ of ordered groups
such that
for each function
$f\in C(X_A,\Z)$
there exists a unitary one-cocycle
$u_t^f \in \U(M(\SOA))$  relative to 
$\rho^{A,f} \otimes\id$
such that 
\begin{equation}
\Phi \circ \Ad(u_t^f) \circ (\rho^{A,f}_t\otimes \id)
 = (\rho^{B,\phi(f)}_t\otimes\id) \circ \Phi
\quad
\text{ for }
 t \in \T. \label{eq:cocyclegauge}
\end{equation}
\end{proposition}
In this section, 
we will first review the proof in \cite{MaMZ2016} of the above proposition to
investigate the K-theoretic behavior of the above isomorphism
$\Phi:\SOA \rightarrow \SOB$.
The proof of the above proposition
is based on the the proof of \cite {MaETDS2004},
in which Morita equivalence of $C^*$-algebras has been used
(cf. \cite{Brown}, \cite{BGR}, \cite{Combes}, \cite{CKRW}, \cite{MaYMJ2007}, \cite{MPT}, \cite{Tomforde}).  
%%%%%%%%%%%%%%%%%%%%%%%%%%%%%%%%%

Suppose that two nonnegative square matrices $A$ and $B$ 
are elementary equivalent
such that $A =CD$ and $B = DC$.
As in the previous section, we may take  and fix bijections
$\varphi_{A,CD}$
from $E_A$ to a subset of $E_C \times E_D$
and 
$\varphi_{B,DC}$
from $E_B$ to a subset of $E_D \times E_C$.
 We set the square matrix
$
Z =
\begin{bmatrix}
0 & C \\
D & 0
\end{bmatrix}
$
as block matrix,
and use the same notation as in the previous sections.

%%%%%%%%%%%%%%%%%%%%%%%%%%%%%%%%%%%%
For an arbitrary fixed 
function $f \in C(X_A,\Z)$,
we may regard it as an element of 
$\DA$ and hence of $\DZ$
by identifying it with 
$f\oplus 0$ in $\DA \oplus \DB = \DZ$.
As
$$
 \exp{(2\pi\sqrt{-1}t (f\oplus 0))}
 =\exp{(2\pi\sqrt{-1}t f)} \oplus P_D
 \in \U(\DZ),
$$
the automorphism
$\rho^{Z,f\oplus 0}_t$ of $\OZ$
for $t \in \T$ defined by \eqref{eq:rhotf}
satisfies
\begin{equation}
\rho^{Z,f\oplus 0}_t(S_c) = \exp{(2\pi\sqrt{-1}t f)}S_c
\quad \text{ for } c \in E_C,\qquad
\rho^{Z,f\oplus 0}_t(S_d) = S_d
\quad \text{ for } d \in E_D. \label{eq:4.2}
\end{equation}
Take 
 $a \in E_A, b \in E_B$ satisfying
$
\varphi_{A,CD}(a) = cd,
\varphi_{B,DC}(b) = dc,
$
The equalities \eqref{eq:4.2} imply
\begin{align*}
\rho^{Z,f\oplus 0}_t(S_c S_d) 
&% =\rho^{Z,f}_t(S_c)\rho^{Z,f}_t(S_d) 
 =\exp{(2\pi\sqrt{-1}t f)} S_c S_d 
 =\rho^{A,f}_t(S_{a}), \\ 
\rho^{Z,f\oplus 0}_t(S_d S_c) 
& %=\rho^{Z,f}_t(S_d)\rho^{Z,f}_t(S_c) 
 = S_d \exp{(2\pi\sqrt{-1}t f)} S_c  
 = S_d \exp{(2\pi\sqrt{-1}t f)} S_d^* S_{b}.
\end{align*} 
We set 
$\phi(f) = \sum_{d\in E_D}S_d f S_d^* \in \DZ$.
As
$P_D \phi(f) P_D = \phi(f)$,
we see that
$\phi(f) \in \DB$ and hence $\phi(f) \in C(X_B,\Z)$ 
which satisfies 
\begin{equation*}
\sum_{d \in E_D}S_d \exp{(2\pi\sqrt{-1}t f)}S_d^*
=
\exp{(2\pi\sqrt{-1}t\phi(f))} \in \U(\DB).
\end{equation*}
We similarly set 
$\psi(g) =\sum_{c \in E_C} S_c g S_c^* \in C(X_A,\Z)$
for $g \in C(X_B,\Z)$.
We thus see the following lemma.
\begin{lemma}[{\cite[Lemma 4.1]{MaMZ2016}}]
For $f \in C(X_A,\Z), g \in C(X_B,\Z)$ and $t \in \T$, 
we have 
\begin{align}
\rho^{Z,f\oplus 0}_t(S_c S_d) =\rho^{A,f}_t(S_{a}), & \qquad
\rho^{Z,f\oplus 0}_t(S_d S_c) =\rho^{B,\phi(f)}_t(S_{b}),
 \label{eq:rhozf} \\
\rho^{Z,0\oplus g}_t(S_d S_c) =\rho^{B,g}_t(S_{b}), & \qquad
\rho^{Z,0\oplus g}_t(S_c S_d) =\rho^{A,\psi(g)}_t(S_{a})
\label{eq:rhozg}
\end{align}
 where
$a \in E_A, b \in E_B$
and
$c \in E_C, d\in E_D$
are satisfying
$
\varphi_{A,CD}(a) = cd
$ and
$
\varphi_{B,DC}(b) = dc$,
respectively.
\end{lemma}
We note that the  homomorphisms
$\phi:C(X_A,\Z) \rightarrow  C(X_B,\Z)$
and
$\psi:C(X_B,\Z) \rightarrow  C(X_A,\Z)$
 satisfy the equalities
\begin{equation}
(\psi \circ \phi)(f) = f \circ \sigma_A,\qquad
(\phi \circ \psi)(g) = g \circ \sigma_B
\end{equation}
for $f \in C(X_A,\Z)$ and $g \in C(X_B,\Z)$
(\cite[Lemma 4.2]{MaMZ2016}).

By \cite[Proposition 4.1]{MaETDS2004},
one may find partial isometries
$v_{A}, v_{B} \in M(\SOZ)$
such that
\begin{equation}
v_{A}^*v_{A} = v_{B}^*v_{B} =1 \otimes 1,
\qquad 
v_{A}v_{A}^* = P_C \otimes 1,
\qquad
v_{B}v_{B}^* = P_D \otimes 1. \label{eq:vAvB}
\end{equation}
Since 
\begin{equation}
\Ad(v_{A}^*):\SOA \rightarrow \SOZ
\quad
\text{ and }
\quad
\Ad(v_{B}^*):\SOB \rightarrow \SOZ \label{eq:isomSO}
\end{equation}
are isomorphisms satisfying
\begin{equation*}
\Ad(v_{A}^*)(\SDA) = \SDZ
\quad
\text{ and }
\quad
\Ad(v_{B}^*)(\SDB) = \SDZ.
\end{equation*}
By putting 
\begin{gather}
w = v_{B}v_{A}^* \in M(\SOZ), \label{eq:wvBvA} \\
\Phi =\Ad(w): \SOA\rightarrow\SOB, \label{eq: Phi} \\
u_t^{A,f} = w^*(\rho^{Z,f\oplus 0}_t \otimes\id)(w) \quad \text{ for } f \in  C(X_A,\Z),\\ u_t^{B,g} = w(\rho^{Z,0\oplus g}_t \otimes\id)(w^*) \quad \text{ for } g \in  C(X_B,\Z), \label{eq:utf}
\end{gather}
they satisfy 
$\Phi(\SDA) = \SDB$ and the equalities
\begin{gather*}
\Phi \circ \Ad(u_t^{A,f}) \circ (\rho^{A,f}_t\otimes \id)
 = (\rho^{B,\phi(f)}_t\otimes\id) \circ \Phi
\quad
\text{ for }
 f \in  C(X_A,\Z), \\
\Phi \circ  (\rho^{A, \psi(g)}_t\otimes \id)
 = \Ad(u_t^{B,g}) \circ (\rho^{B,g}_t\otimes\id) \circ \Phi
\quad
\text{ for }
g \in  C(X_B,\Z).
\end{gather*}

\medskip

The above discussion is a sketch of the proof of Proposition \ref{prop:4.3}
given in \cite{MaMZ2016}.
%In the rest of this section, we will investigate the K-theoretic behavior of the map 
%$\Phi: \SOA \rightarrow \SOB$. 

%%%%%%%%%%%%%%%%%%%%%%%%%%%%%%%%%%%%%%%%
%%%%%%%%%%%%%%%%%%%%%%%%%%%%%%%%%%%%%%%%%
%%\section{ Cocycle }
%%%%%%%%%%%%%%%%%%%%%%%%%%%%%%%%%%
In what follows, we will reconstruct partial isometries
$v_A, v_B$ satisfying  \eqref{eq:vAvB}
to  investigate the K-theoretic behavior of the map 
$\Phi: \SOA \rightarrow \SOB$
in the following section.
The idea of the reconstruction is due to
the proof of \cite[Lemma 2.5]{Brown} (cf. \cite[Proposition 4.1]{MaETDS2004}).

We are assuming that $A=CD, B= DC$. 
Keep the notations as in the preceding section.
Put 
$E_C= \{ c_1, \dots,c_{N_C}\}$
and
$E_D= \{ d_1, \dots,d_{N_D}\}$
for the matrices $C$ and $D$ respectively.
For $k =1,\dots,N_D$, take $c(k) \in E_C$ such that 
$c(k) d_k \in B_2(X_Z)$ so that we have
\begin{equation}
S_{c(k)}^*S_{c(k)} \ge S_{d_k} S_{d_k}^*. \label{eq:ckd}
\end{equation}
Similarly 
for $l =1,\dots,N_C$, take $d(l) \in E_D$ such that 
$d(l) c_l \in B_2(X_Z)$ so that we have
\begin{equation}
S_{d(l)}^*S_{d(l)} \ge S_{c_l} S_{c_l}^*. \label{eq:dlc}
\end{equation}
Put
\begin{align}
U_0 = P_C, \qquad  U_k & = S_{c(k)} S_{d_k} S_{d_k}^* 
\quad \text{ for } k=1,\dots,N_D, \label{eq:ak}\\
T_0 = P_D, \qquad  T_l & = S_{d(l)} S_{c_l} S_{c_l}^* 
\quad \text{ for } l=1,\dots,N_C. \label{eq:bl}
\end{align}
We then have
\begin{align}
\sum_{k=1}^{N_D} U_k^* U_k 
& =  \sum_{k=1}^{N_D}S_{d_k} S_{d_k}^* S_{c(k)} ^*S_{c(k)} S_{d_k} S_{d_k}^*
   =  \sum_{k=1}^{N_D}S_{d_k} S_{d_k}^* =P_D, \label{eq:akpd} \\
\sum_{k=1}^{N_C} T_l^* T_l 
& =  \sum_{l=1}^{N_C}S_{c_l} S_{c_l}^* S_{d(l)} ^*S_{d(l)} S_{c_l} S_{c_l}^*
   =  \sum_{l=1}^{N_C}S_{c_l} S_{c_l}^* =P_C. \label{eq:blpc} 
\end{align} 
We decompose the set  $\N$ of natural numbers
into disjoint infinite subsets $\N = \cup_{j=1}^{\infty} {\N}_j$,
and decompose $\N_j$ for each $j$ once again into disjoint infinite sets
$\N_j = \cup_{k=0}^{\infty} {\N}_{j_k}.$
Let $\{ e_{i,j} \}_{i,j \in \N}$ be a set of matrix units which generate the algebra 
${\K} = \K(\ell^2(\N)).$
Put the projections
$f_j = \sum_{i\in {\N}_j} e_{i,i}$ 
and
$f_{j_k} = \sum_{i\in {\N}_{j_k}} e_{i,i}.$
Take a partial isometry
$s_{j_k,j}$ such that
$
s_{j_k,j}^*s_{j_k,j} = f_j,
s_{j_k,j}s_{j_k,j}^* = f_{j_k}
$
and put $s_{j,j_k} =s_{j_k,j}^*$.
We set for $n=1,2,\dots, $
\begin{align*}
u_n = \sum_{k=1}^{N_D} U_k \otimes s_{n_k,n}, 
& \qquad w_n = P_C \otimes s_{{n_0},n} +  u_n,\\ 
t_n = \sum_{l=1}^{N_C} T_l \otimes s_{n_l,n},
& \qquad z_n =  P_D \otimes s_{{n_0},n}  + t_n. 
\end{align*}
Then we have
\begin{lemma}
Keep the above notations.
\begin{enumerate}
\renewcommand{\theenumi}{\roman{enumi}}
\renewcommand{\labelenumi}{\textup{(\theenumi)}}
\item$w_n^* w_n = 1\otimes f_n$ and $w_n w_n^* \le P_C \otimes f_n$.
\item$z_n^* z_n = 1\otimes f_n$ and $z_n z_n^* \le P_D \otimes f_n$.
\end{enumerate}
\end{lemma}
\begin{proof}
(i)
Since $u_n^* u_n = P_D \otimes f_n$, we have
$$
w_n^* w_n = P_C \otimes f_{n} + u_n^* u_n = P_C \otimes f_n + P_D \otimes f_n 
                = 1 \otimes f_n. 
$$
On the other hand, we know that 
$u_n(P_C \otimes s_{n,n_0}) =(P_C \otimes s_{n,n_0}) u_n^* =0$
so that we have 
%\begin{align*}
$$
w_n w_n^*
 = P_C \otimes f_{n_0} + u_n u_n^* 
 = P_C \otimes f_{n_0} +
\sum_{k=1}^{N_D} S_{c(k)} S_{d_k}S_{d_k}^*S_{c(k)}^* \otimes f_{n_k}. 
$$
As $f_{n_0}, f_{n_k} \le f_n$,
we have 
$$
w_n w_n^* \le P_C \otimes f_n.
$$
(ii) is similarly shown to (i).
\end{proof}
We will reconstruct and study the unitary $v_A$
 in \eqref{eq:vAvB}.
Let $f_{n,m}$ be a partial isometry satisfying
$f_{n,m}^*f_{n,m} = f_m,\, f_{n,m}f_{n,m}^* = f_n.$
We put
\begin{align*}
v_1 &  = w_1 = P_C \otimes s_{1_0,1} + u_1, \\
v_{2n} & = (P_C \otimes f_n - v_{2n-1}v_{2n-1}^*)(P_C \otimes f_{n,n+1})
\quad \text{ for } 1\le n \in \N, \\
%v_2 & = (P_C \otimes f_1 - v_1 v_1^*)(P_C \otimes f_{1,2}), \\
v_{2n-1} & = w_n(1\otimes f_n - v_{2n-2}^*v_{2n-2}) \quad \text{ for } 2\le n \in \N. 
%v_{2n} & = (P_C \otimes f_n - v_{2n-1}v_{2n-1}^*)(P_C \otimes f_{n,n+1})
\end{align*} 
\begin{lemma}Keep the above notations.
\begin{enumerate}
\renewcommand{\theenumi}{\roman{enumi}}
\renewcommand{\labelenumi}{\textup{(\theenumi)}}
\item$v_{2n-2}^* v_{2n-2} + v_{2n-1}^* v_{2n-1} = 1\otimes f_n$.
\item$v_{2n-1} v_{2n-1}^* + v_{2n} v_{2n}^* = P_C\otimes f_n$.
\end{enumerate}
\end{lemma}
\begin{proof}
(i) As $w_n^* w_n = 1 \otimes f_n$,
we have
\begin{align*}
& v_{2n-2}^* v_{2n-2} + v_{2n-1}^* v_{2n-1} \\
= & v_{2n-2}^* v_{2n-2} + (1 \otimes f_n -v_{2n-2}^*v_{2n-2})w_n^*w_n(1 \otimes f_n -v_{2n-2}^*v_{2n-2}) \\
= & v_{2n-2}^* v_{2n-2} + 1 \otimes f_n -v_{2n-2}^*v_{2n-2} \\
 = & 1\otimes f_n.
\end{align*}

(ii) We have 
\begin{align*}
   & v_{2n-1} v_{2n-1}^* + v_{2n} v_{2n}^* \\
= & v_{2n-1} v_{2n-1}^* + (P_C\otimes f_n - v_{2n-1} v_{2n-1}^*)(P_C\otimes f_{n}) 
    (P_C\otimes f_n - v_{2n-1} v_{2n-1}^*) \\
= & v_{2n-1} v_{2n-1}^* + P_C\otimes f_n - v_{2n-1} v_{2n-1}^* \\
= & P_C\otimes f_n.
\end{align*}
\end{proof}
%By the above lemma, one may see that the summation
%$\sum_{n=1}^\infty v_{n}$ converges in $M(\OZ\otimes\K)$ 
%to certain partial isometry written $v_A$
% in the strict topology of the multiplier algebra of $\OZ\otimes \K$.

By the above lemma, one may see that 
the summations
$\sum_{n=1}^\infty v_{2n-2}$
and
$\sum_{n=1}^\infty v_{2n-1}$
 converge in $M(\OZ\otimes\K)$ 
to certain partial isometries written $v_{ev}$ and $v_{od}$ respectively
 in the strict topology of the multiplier algebra of $\OZ\otimes \K$.
Similarly we obtain a partial isometry
$v_A =\sum_{n=1}^\infty v_{n}$
in $M(\OZ\otimes \K)$ in the strict topology.
Therefore we have the next lemma.
\begin{lemma}
The partial isometries $v_{ev}, v_{od} $ and $v_A$ defined above satisfy the following relations:
\begin{enumerate}
\renewcommand{\theenumi}{\roman{enumi}}
\renewcommand{\labelenumi}{\textup{(\theenumi)}}
\item$v_A = v_{od} + v_{ev}.$
\item$v_{od}^* v_{od} + v_{ev}^* v_{ev} = 1\otimes 1.$
\item$v_{od} v_{od}^* + v_{ev} v_{ev}^* = P_C\otimes 1.$
\item$v_A^*v_A = 1 \otimes 1$ and $ v_A v_A^* = P_C \otimes 1$.
\end{enumerate}
\end{lemma}
We put
\begin{equation*}
q_{od}^C  = \sum_{n=1}^\infty v_{2n-1}(P_C\otimes 1) v_{2n-1}^*,
 \qquad
q_{od}^D  = \sum_{n=1}^\infty v_{2n-1}(P_D\otimes 1) v_{2n-1}^*
\end{equation*}
so that 
\begin{equation*}
q_{od}^C + q_{od}^D = v_{od} v_{od}^* \quad 
\text{ and hence }
\quad 
q_{od}^C + q_{od}^D + v_{ev} v_{ev}^* = P_C \otimes 1. %\label{eq:qod}
\end{equation*}
We will show the following lemma. 
\begin{lemma}\label{lem:3.4}
 $v_A(\rho^{Z,f\oplus 0}_t\otimes \id)(v_A^*) 
=q_{od}^C + (U_t(-f)\otimes 1) q_{od}^D + v_{ev} v_{ev}^*.
$
\end{lemma}
\begin{proof}
We notice that 
$\rho^{Z,f\oplus 0}_t(S_c) = U_t(f)S_c$ for $c \in E_C$
 and 
$\rho^{Z,f\oplus 0}_t(S_d) = S_d$ for $d \in E_D.$
As $v_{2n-1} v_{2n-1}^* \in D_Z\otimes{\mathcal{C}}$ 
so that 
$(\rho^{Z,f\oplus 0}_t\otimes\id)(v_{2n-1} v_{2n-1}^*) = v_{2n-1} v_{2n-1}^*$ 
and hence
$(\rho^{Z,f\oplus 0}_t\otimes \id)(v_{ev}) = v_{ev}$.
We then have
\begin{align*}
v_A(\rho^{Z,f\oplus 0}_t\otimes \id)(v_A^*)
 & =  v_{od}(\rho^{Z,f\oplus 0}_t\otimes \id)(v_{od}^*)  
     + v_{ev}(\rho^{Z,f\oplus 0}_t\otimes \id)(v_{ev}^*) \\
 & =  \sum_{n=1}^\infty v_{2n-1}(\rho^{Z,f\oplus 0}_t\otimes \id)(v_{2n-1}^*)  
     + v_{ev} v_{ev}^*.
\end{align*}
Since
\begin{equation*}
v_1(P_C \otimes 1)   = P_C \otimes s_{1_0,1} \quad
\text{ and } \quad
v_1(P_D\otimes 1 )  = \sum_{k=1}^{N_D}S_{c(k)}S_{d_k}S_{d_k}^*\otimes s_{1_k,1},
\end{equation*}
we have
\begin{align*}
(\rho^{Z,f\oplus 0}_t\otimes \id)(v_{1}^*)
& = (P_C \otimes 1)v_1^* + (\rho^{Z,f\oplus 0}_t\otimes \id)((P_D\otimes 1)v_1^* ) \\
& = (P_C \otimes 1)v_1^* 
+ \sum_{k=1}^{N_D}S_{d_k}S_{d_k}^* \rho^{Z,f\oplus 0}_t(S_{c(k)}^*)\otimes s_{1_k,1}^* \\
& = (P_C \otimes 1)v_1^* 
+ \sum_{k=1}^{N_D}S_{d_k}S_{d_k}^* S_{c(k)}^*U_t(-f) \otimes s_{1_k,1}^* \\
& = (P_C \otimes 1)v_1^* 
 +(P_D \otimes 1)v_1^* (U_t(-f) \otimes 1),
\end{align*}
so that 
\begin{align*}
v_1(\rho^{Z,f\oplus 0}_t\otimes \id)(v_{1}^*)
& = v_1 (P_C \otimes 1)v_1^* 
 +v_1 (P_D \otimes 1)v_1^* (U_t(-f) \otimes 1) \\
& = v_1 (P_C \otimes 1)v_1^* 
 + (U_t(-f) \otimes 1) v_1 (P_D \otimes 1)v_1^*.  \\
\end{align*}
For $2\le n \in \N$, 
we have
\begin{align*}
v_{2n-1}(P_C \otimes 1) 
& = (P_C \otimes s_{n_0,n})(1 \otimes f_n - v_{2n-2}^* v_{2n-2}), \\ 
v_{2n-1}(P_D \otimes 1) 
& = \sum_{k=1}^{N_D} (S_{c(k)}S_{d_k}S_{d_k}^*\otimes s_{n_k,n})(1 \otimes f_n - v_{2n-2}^* v_{2n-2}), 
\end{align*}
and hence 
\begin{align*}
(\rho^{Z,f\oplus 0}_t\otimes \id)((P_D \otimes 1)v_{2n-1}^*)
& = (1 \otimes f_n - v_{2n-2}^* v_{2n-2})
       \sum_{k=1}^{N_D}S_{d_k}S_{d_k}^* \rho^{Z,f\oplus 0}_t(S_{c(k)}^*)\otimes s_{n_k,n}^* \\
& = (1 \otimes f_n - v_{2n-2}^* v_{2n-2})
         \sum_{k=1}^{N_D}S_{d_k}S_{d_k}^* S_{c(k)}^*U_t(-f) \otimes s_{n_k,n}^* \\
& = (P_D\otimes 1) v_{2n-1}^*(U_t(-f) \otimes 1)
\end{align*}
so that 
\begin{align*}
v_{2n-1}(\rho^{Z,f\oplus 0}_t\otimes \id)(v_{2n-1}^*)
& = v_{2n-1} (P_C \otimes 1)v_{2n-1}^* 
 +v_{2n-1} (P_D \otimes 1)v_{2n-1}^* (U_t(-f) \otimes 1) \\
& = v_{2n-1} (P_C \otimes 1)v_{2n-1}^* 
 + (U_t(-f) \otimes 1) v_{2n-1} (P_D \otimes 1)v_{2n-1}^*.  \\
\end{align*}
Therefore we have
\begin{equation*}
v_{od} (\rho^{Z,f\oplus 0}_t\otimes\id)(v_{od}^*) 
= q_{od}^C + (U_t(-f)\otimes 1) q_{od}^D
\end{equation*}
and hence
\begin{equation*}
v_A(\rho^{Z,f\oplus 0}_t\otimes \id)(v_A^*) 
=q_{od}^C + (U_t(-f)\otimes 1) q_{od}^D + v_{ev} v_{ev}^*.
\end{equation*}
\end{proof}
%%%%%%%%%%%%%%%%%%%%%%%%%%%%%%%

By using $t_n,  z_n$ instead of 
$u_n, w_n$ respectively, 
we similarly obtain a partial isometry
$v_B$
in $M(\OZ\otimes \K)$ in the strict topology.
We then see the following lemmas.
\begin{lemma} \label{lem:3.5} \hspace{6cm}
\begin{enumerate}
\renewcommand{\theenumi}{\roman{enumi}}
\renewcommand{\labelenumi}{\textup{(\theenumi)}}
\item The partial isometry $v_A(\rho_t^{Z, f\oplus 0} \otimes \id)(v_A^*)$
for $f \in C(X_A,\Z), t \in \T$ belongs to $M(\SDA)$ and satisfies
\begin{equation}
v_A(\rho_t^{Z, {(f_1 + f_2)}\oplus 0} \otimes \id)(v_A^*) 
= v_A(\rho_t^{Z, {f_1}\oplus 0} \otimes \id)(v_A^*)
   v_A(\rho_t^{Z, {f_2}\oplus 0} \otimes \id)(v_A^*) \label{eq:vAf}
\end{equation}
for $f_1, f_2 \in C(X_A,\Z), t \in \T. $
\item
The partial isometry $v_B(\rho_t^{Z, 0\oplus g} \otimes \id)(v_B^*)$ 
for $g \in C(X_B,\Z), t \in \T$ belongs to $M(\SDB)$ and satisfies
\begin{equation}
v_B(\rho_t^{Z, 0\oplus{(g_1 + g_2)}} \otimes \id)(v_B^*) 
= v_B(\rho_t^{Z, 0\oplus{g_1}} \otimes \id)(v_B^*)
   v_B(\rho_t^{Z, 0\oplus{g_2}} \otimes \id)(v_B^*)
\label{eq:vBg}
\end{equation}
 for
$  g_1,g_2 \in C(X_B,\Z), t \in \T. 
$
\end{enumerate}
\end{lemma}
\begin{proof}
(i) Since the projections 
$q_{od}^C, q_{od}^D, v_{ev} v_{ev}^*$ are all belong to the multiplier algebra $M(\SDA)$
of $\SDA$,
the preceding lemma ensures us
that the partial isometry 
$v_A(\rho^{Z, f\oplus 0} \otimes \id)(v_A^*)$ belongs to $M(\SDA)$.
As $U_t(f_1 + f_2) = U_t(f_1)U_t(f_2)$,  the equality \eqref{eq:vAf} follows.

(ii) is similarly shown to (i).
\end{proof}

\begin{lemma} \label{lem:3.6} \hspace{6cm}
\begin{enumerate}
\renewcommand{\theenumi}{\roman{enumi}}
\renewcommand{\labelenumi}{\textup{(\theenumi)}}
\item$(\rho_t^{Z, 0\oplus g} \otimes \id)(v_A) = v_A $ for $g \in C(X_B,\Z), t \in \T.$
\item$(\rho_t^{Z, f\oplus 0} \otimes \id)(v_B) = v_B $ for $f \in C(X_A,\Z), t \in \T.$
\end{enumerate}
\end{lemma}
\begin{proof}
(i)
Since
$\rho_t^{Z,0\oplus g}(S_c) = S_c, \rho_t^{Z,0\oplus g}(S_d) = e^{2 \pi \sqrt{-1}t g}S_d,$
we have
\begin{equation*}
\rho_t^{Z,0\oplus g}(U_k) 
= \rho_t^{Z,0\oplus g}(S_{c(k)} S_{d_k}S_{d_k}^*)
= S_{c(k)} e^{2 \pi \sqrt{-1}t g}S_{d_k}S_{d_k}^*e^{-2 \pi \sqrt{-1}t g})
= S_{c(k)} S_{d_k}S_{d_k}^* =U_k.
\end{equation*} 
Hence 
$(\rho_t^{Z, 0\oplus g} \otimes \id)(u_n) = u_n $
so that 
$(\rho_t^{Z, 0\oplus g} \otimes \id)(w_n) = w_n. $
We then have
\begin{equation*}
(\rho_t^{Z, 0\oplus g} \otimes \id)(v_1) 
=(\rho_t^{Z, 0\oplus g} \otimes \id)(P_C \otimes s_{1_0,1} +u_1)  
=P_C \otimes s_{1_0,1} +u_1
=v_1.
\end{equation*}
Since
$v_{2n-1}v_{2n-1}^*, v_{2n-2}^*v_{2n-2} \in \SDZ$ 
and
the restriction of $\rho_t^{Z,0\oplus g}\otimes \id$ to
$\SDZ$ is the identity,
we easily know that 
\begin{equation*}
(\rho_t^{Z, 0\oplus g} \otimes \id)(v_{2n})
= v_{2n}, \qquad
(\rho_t^{Z, 0\oplus g} \otimes \id)(v_{2n-1})
= v_{2n-1}
\quad \text{ for }
n \in \N.
\end{equation*}
We thus have 
$(\rho_t^{Z, 0\oplus g} \otimes \id)(v_n)
= v_n
$ for all $n \in \N$ and hence
$(\rho_t^{Z, 0\oplus g} \otimes \id)(v_A)
= v_A.
$

(ii) is similarly shown to (i).
\end{proof}

We put 
\begin{gather}
w = v_{B}v_{A}^* \in M(\SOZ), \label{eq:wvBvA} \\
%\Phi =\Ad(w): \SOA\rightarrow\SOB, \label{eq: Phi} \\
u_t^{A,f} = w^*(\rho^{Z,f\oplus 0}_t \otimes\id)(w) \quad \text{ for } f \in  C(X_A,\Z),\\ u_t^{B,g} = w(\rho^{Z,0\oplus g}_t \otimes\id)(w^*) \quad \text{ for } g \in  C(X_B,\Z).\label{eq:utf}
\end{gather}
By Lemma \ref{lem:3.6},
we have
\begin{equation}
u_t^{A,f} 
= v_A v_B^*(\rho^{Z,f\oplus 0}_t \otimes\id)(v_B)(\rho^{Z,f\oplus 0}_t \otimes\id)(v_A^*)
= v_A(\rho^{Z,f\oplus 0}_t \otimes\id)(v_A^*) \label{eq:utafva}
\end{equation}
and similarly 
$u_t^{B,g} = v_B(\rho^{Z,0\oplus g}_t \otimes\id)(v_B^*).$ 
\begin{lemma} \hspace{6cm}
\begin{enumerate}
\renewcommand{\theenumi}{\roman{enumi}}
\renewcommand{\labelenumi}{\textup{(\theenumi)}}
\item For each $f \in C(X_A,\Z)$,
the partial isometries 
$u_t^{A,f}, t \in \T$ give rise to a unitary representation of $\T$ in $M(\SDA)$
 and satisfies
$u_t^{A,f_1 + f_2} = u_t^{A,f_1} u_t^{A,f_2}$
for $f_1, f_2 \in C(X_A,\Z). $
\item For each $g \in C(X_B,\Z)$,
the partial isometries 
$u_t^{B,g}, t \in \T$ give rise to a unitary representation of $\T$ in $M(\SDB)$
 and satisfies
$u_t^{B,g_1 + g_2} = u_t^{B,g_1} u_t^{B,g_2}$
for $g_1, g_2 \in C(X_B,\Z). $
\end{enumerate}
\end{lemma}
\begin{proof}
(i)
By Lemma \ref{lem:3.4} and \eqref{eq:utafva}, we have
\begin{align*}
u_t^{A,f} u_s^{A,f}
= & v_A(\rho^{Z,f\oplus 0}_t\otimes \id)(v_A^*) 
      v_A(\rho^{Z,f\oplus 0}_s\otimes \id)(v_A^*) \\
= & (q_{od}^C + (U_t(-f)\otimes 1) q_{od}^D + v_{ev} v_{ev}^*)
      (q_{od}^C + (U_s(-f)\otimes 1) q_{od}^D + v_{ev} v_{ev}^*) \\
= & q_{od}^C + (U_{t+s}(-f)\otimes 1) q_{od}^D + v_{ev} v_{ev}^* 
=  u_{t+s}^{A,f}.
\end{align*}
The equality
$u_t^{A,f_1 + f_2} = u_t^{A,f_1} u_t^{A,f_2}$
immediately follows from Lemma \ref{lem:3.5}.
(ii) is similarly shown to (i).
\end{proof}

We thus have
\begin{proposition}\label{prop:main1}
Let $A , B$ 
be nonnegative irreducible and non-permutation matrices.
Suppose that $A = CD, \, B = DC$ for some nonnegative rectangular matrices $C, D$.
Then  
 there exist an isomorphism
$\Phi:\SOA \rightarrow \SOB$ satisfying 
$\Phi(\SDA) = \SDB$, 
and unitary representations 
$t \in \T \rightarrow u^{A,f}_t \in M(\SDA)$ for each $f \in C(X_A,\Z)$
and
$t \in \T \rightarrow u^{B,g}_t \in M(\SDB)$ for each $g \in C(X_B,\Z)$
such that
\begin{gather}
\Phi \circ \Ad(u_t^{A,f}) \circ (\rho^{A,f}_t\otimes \id)
 = (\rho^{B,\phi(f)}_t\otimes\id) \circ \Phi
\quad
\text{ for }
 f \in  C(X_A,\Z), \label{eq:3.12} \\
\Phi \circ  (\rho^{A, \psi(g)}_t\otimes \id)
 = \Ad(u_t^{B,g}) \circ (\rho^{B,g}_t\otimes\id) \circ \Phi
\quad
\text{ for }
g \in  C(X_B,\Z). \label{eq:3.13}
\end{gather}
\end{proposition}
\begin{proof}
As in the proof of \cite[Proposition 4.3]{MaMZ2016},
the map 
$\Phi = \Ad(w)$ where $w = v_B v_A^*$ gives rise to 
an isomorphism
$\Phi:\SOA \rightarrow \SOB$
such that 
$ \Phi(\SDA) = \SDB$ and
\begin{equation*}
\Phi \circ \Ad(u_t^{A,f}) \circ (\rho^{A,f}_t\otimes \id)
=  (\rho^{B,\phi(f)}_t \otimes\id)\circ \Phi.
\end{equation*} 
The other equality \eqref{eq:3.13}
is similarly shown to (i).
\end{proof}
Since both the homomorphisms 
$\varphi: C(X_A,\Z) \rightarrow C(X_B,\Z)$
and
$\psi: C(X_B,\Z) \rightarrow C(X_A,\Z)$
satisfy 
$\varphi(1) = 1, \psi(1) =1$,
we have the following corollary. 
\begin{corollary}[cf. {\cite[3.8 Theorem]{CK}, \cite[2.3 Theorem]{Cu3}}] \label{cor:main}
Let $A, B$ be irreducible non-permutation matrices.
Suppose that two-sided topological Markov shifts
$(\bar{X},\bar{\sigma}_A)$ and $(\bar{X}_B,\bar{\sigma}_B)$ are topologically conjugate.
Then  there exist an isomorphism 
$\Phi:\SOA \rightarrow \SOB $ of $C^*$-algebras
satisfying
$\Phi(\SDA) = \SDB$,
and  
unitary representations
 $t \in \T \rightarrow u_t^A \in M(\SDA)$
and
 $t \in \T \rightarrow u_t^B \in M(\SDB)$
 such that
\begin{gather*}
\Phi \circ \Ad(u^A_t) \circ (\rho^{A}_t \otimes\id) 
= (\rho^{B}_t \otimes \id) \circ \Phi, \\
\Phi \circ (\rho^{A}_t \otimes\id) 
= \Ad(u^B_t) \circ (\rho^{B}_t \otimes \id) \circ \Phi
\end{gather*}
where $\rho^A_t$ and $\rho^B_t$ are the gauge actions on $\OA$ and $\OB$, respectively.
\end{corollary}
\begin{remark}
 We must emphasize that Cuntz-- Krieger in \cite[3.8 Theorem]{CK} 
and  Cuntz in \cite[2.3 Theorem]{Cu3}
 have shown that the stabilized 
Cuntz--Krieger triplet
$(\SOA,\SDA,\rho^A\otimes\id)$
is invariant under topological conjugacy of the two-sided topological Markov shifts
$(\bar{X}_A,\bar{\sigma}_A)$. 
Hence the above corollary is weaker than their  result.
\end{remark}

Before ending this section, 
we will introduce a notion of strong Morita equivalence 
in the stabilized Cuntz--Krieger triplets.
The triplet
$(\SOA, \SDA, \rho^A\otimes\id)$
is called the stabilized  Cuntz--Krieger triplets.
Two stabilized  Cuntz--Krieger triplets
$(\SOA,\SDA,\rho^A\otimes\id)$ and
$(\SOB,\SDB,\rho^B\otimes\id)$ 
are said to be {\it strong Morita equivalent in} $1$-{\it step}
if 
there exists a stabilized Cuntz--Krieger triplet 
$(\SOZ,\SDZ,\rho^Z\otimes\id)$ such that 
%%%%%%%%%%%%%%%%%%%%%%%%%
%\begin{enumerate}
%\item 
there exist isomorphisms of $C^*$-algebras
\begin{equation*}
\Phi_A : \SOZ \longrightarrow \SOA, \qquad
\Phi_B : \SOZ \longrightarrow \SOB 
\end{equation*}
satisfying
\begin{align*}
\Phi_A (\DZ\otimes&\C ) = \SDA, \qquad
\Phi_B (\SDZ)= \SDB, \\ 
\rho^Z_t\otimes\id 
=&
 (\Phi_B^{-1}\circ \rho^B_t\otimes\id\circ\Phi_B) \circ
(\Phi_A^{-1}\circ \rho^A_t\otimes\id\circ\Phi_A) \\
=&(\Phi_A^{-1}\circ \rho^A_t\otimes\id\circ\Phi_A) \circ
(\Phi_B^{-1}\circ \rho^B_t\otimes\id\circ\Phi_B).
\end{align*}
%\end{enumerate}
If two stabilized Cuntz--Krieger triplets
$(\SOA,\SDA,\rho^A\otimes\id)$ and
$(\SOB,\SDB,\rho^B\otimes\id)$ 
are connected by $n$-chains of 
strong Morita equivalences in $1$-step, 
they are said to be 
strong Morita equivalent in $n$-step,
or simply 
strong Morita equivalent.
\begin{proposition}
Suppose that $A, B$ are elementary equivalent such that
$A= CD, B = DC$. 
Then the
stabilized  Cuntz--Krieger triplets
$(\SOA,\SDA,\rho^A\otimes\id)$ and
$(\SOB,\SDB,\rho^B\otimes\id)$ 
are strong Morita equivalent in $1$-step.
\end{proposition}
\begin{proof}
Let 
$Z =
\begin{bmatrix}
0 & C\\
D & 0
\end{bmatrix}.
$
Take partial isometries 
$v_A,v_B \in M(\SOZ)$ satisfying
\eqref{eq:vAvB}.
By Lemma \ref{lem:3.6},
the following identities hold
\begin{equation*}
(\rho_t^{Z, 0\oplus 1} \otimes \id)(v_A) = v_A,\qquad
(\rho_t^{Z, 1\oplus 0} \otimes \id)(v_B) = v_B.
\end{equation*} 
Define
$
\Phi_A = \Ad(v_A),
\Phi_B = \Ad(v_B).
$
As in \eqref{eq:isomSO},
they give rise to isomorphisms
$$
\Phi_A : \SOZ \longrightarrow \SOA, \qquad
\Phi_B : \SOZ \longrightarrow \SOB 
$$
satisfying
$$
\Phi_A (\SDZ) = \SDA, \qquad
\Phi_B (\SDZ)= \SDB.
$$
Since we see
\begin{gather*}
\rho_t^{Z, 0\oplus 1}(S_c) = S_c,\qquad
\rho_t^{Z, 0\oplus 1}(S_d) = e^{2\pi\sqrt{-1}t}S_d,\\
\rho_t^{Z, 1\oplus 0}(S_c) = e^{2\pi\sqrt{-1}t}S_c,\qquad
\rho_t^{Z, 1\oplus 0}(S_d) = S_d
\end{gather*}
for $c \in C, d \in D$, we have
for $x\otimes K \in \SOZ$
\begin{align*}
((\rho_t^{A} \otimes \id) \circ\Phi_A)(x\otimes K)
=& (\rho_t^{Z, 0\oplus 1} \otimes \id) (v_A(x\otimes K)v_A^*) \\
=& v_A(\rho_t^{Z, 0\oplus 1}\otimes \id)(x\otimes K) v_A^* \\
=& \Phi_A \circ (\rho_t^{Z, 0\oplus 1} \otimes \id)(x \otimes K).
\end{align*} 
Hence we have 
$(\rho_t^{A} \otimes \id) \circ\Phi_A =\Phi_A \circ (\rho_t^{Z, 0\oplus 1} \otimes \id)$
and similarly
$(\rho_t^{B} \otimes \id) \circ\Phi_B =\Phi_B \circ (\rho_t^{Z, 1\oplus 0} \otimes \id)$.
Since
$\rho_t^{Z}\otimes \id
=(\rho_t^{Z, 1\oplus 0}\otimes \id) \circ (\rho_t^{Z, 0\oplus 1}\otimes \id)
=(\rho_t^{Z, 0\oplus 1}\otimes \id)\circ (\rho_t^{Z, 1\oplus 0}\otimes \id),
$
we know the assertion.
\end{proof}
Therefore we have the following corollary.
\begin{corollary}
If  $A, B$ are strong shift equivalent, 
then
the stabilized  Cuntz--Krieger triplets
$(\SOA,\SDA,\rho^A\otimes\id)$ and
$(\SOB,\SDB,\rho^B\otimes\id)$ 
are strong Morita equivalent.
\end{corollary}

%%%%%%%%%%%%%%%%%%%%%%%%%%%%%%%%%%%%%%
%%%%%%%%%%%%%%%%%%%%%%%%%%%%%%%%%%
\section{Behavior on K-theory}
%%%%%%%%%%%%%%%%%%%%%%%%%%%%%%%%%%%%%%%%%%%%
In this section we will study the behavior of the isomorphism 
$\Phi:\SOA \rightarrow \SOB$ in Proposition \ref{prop:main1} on their K-groups 
$\Phi_* : K_0(\OA) \rightarrow K_0(\OB)$
under the condition $A=CD, B=DC$.  

Recall that $A =[A(i,j)]_{i,j=1}^N $ is an $N\times N$ matrix with entries in nonnegative integers.
Then the associated graph $G_A = (V_A,E_A)$ consists of the vertex set 
$V_A =\{ v^A_1, \dots, v_N^A\}$ of $N$ vertices and edge set 
$E_A =\{a_1,\dots,a_{N_A} \}$, where 
there  are $A(i,j)$ edges  from $v_i^A$ to $v_j^A$.
% Hence the total number of edges is $\sum_{i,j=1}^N A(i,j)$ denoted by $N_A$.
Denote by $t(a_i),  s(a_i)$ the terminal vertex  of $a_i$, 
the source vertex of $a_i$, respectively. 
The graph $G_A$ has the $N_A \times N_A$
 transition matrix $A^G =[A^G(i,j)]_{i,j=1}^{N_A}$ of edges 
defined by \eqref{eq:AG}.
The Cuntz--Krieger algebra $\OA$ is defined as the Cuntz--Krieger algebra
 ${\mathcal{O}}_{A^G}$ 
for the matrix $A^G$ which is the universal $C^*$-algebra generated by 
partial isometries
$S_{a_i}, i=1,\dots, N_A$ subject to the relations \eqref{eq:OAG}.
We similarly consider the $N_B \times N_B$ matrix $B^G$ with entries in $\{0,1\}$
for the graph $G_B = (V_B, E_B)$ of the matrix $B$
with
vertex set $V_B =\{ v_1^B,\dots,v_M^{B} \}$
and edge set $E_B = \{ b_1, \dots,b_{N_B}\}$,
so that we have the other Cuntz--Krieger algebra ${\mathcal{O}}_{B^G}$
for the matrix $B^G$ which is denoted by $\OB.$

Now we are assuming that $A=CD$ and $B= DC$ for some nonnegative rectangular matrices $C$ and $D$.
Both $A$ and $B$ are also assumed to be irreducible and not any permutations. 
Since $A=CD$, the edge set $E_A$ is regarded as a subset of the product
$E_C \times E_D$ of those of $E_C$ and $E_D$.
As in Section 2, we may take a bijection 
$\phi_{A,CD}$ from $ E_A$ to a subset of $E_C\times E_D$. 
 For any $a_i \in E_A$, there uniquely exist
$c(a_i) \in E_C$ and $d(a_i) \in E_D$ such that 
$\phi_{A,CD}(a_i) = c(a_i) d(a_i)$.
We write it simply as $a_i = c(a_i)d(a_i)$. 
Similarly, for any edge $b_l \in E_B$, 
there uniquely exist  
$d(b_l) \in E_D$ and $c(b_l) \in E_C$ such that
$\phi_{B,DC}(b_l) = d(b_l) c(b_l)$, simply written 
$b_l = d(b_l)c(b_l).$
We define $N_A\times N_B$ matrix 
$\hat{D} = [\hat{D}(i,l)]_{i=1,\dots,N_A}^{l=1,\dots,N_B}$ by
\begin{equation}
\hat{D}(i,l) =
\begin{cases} 
1 &  \text{  if  } d(a_i) = d(b_l), \\
0 &  \text{  otherwise.}
\end{cases} \label{eq:Dhat}
\end{equation}
\begin{lemma}
The matrix $\hat{D}^t: \Z^{N_A}\rightarrow \Z^{N_B}$ 
induces a homomorphism from
$\Z^{N_A}/{(\id - {(A^G)}^{t})\Z^{N_A}}$
to 
$\Z^{N_B}/{(\id - {(B^G)}^{t})\Z^{N_B}}$
as abelian groups. 
\end{lemma}
\begin{proof}
For $i=1,\dots,N_A$ and $l=1,\dots,N_B$,
we know that both
$$
[A^G\hat{D}](i,l) = \sum_{j=1}^{N_A}A^G(i,j)\hat{D}(j,l)
\quad \text{ and } \quad
[\hat{D} B^G](i,l) = \sum_{k=1}^{N_B} \hat{D}(i,k) B^G(k,l)
$$
are the cardinal number of the set 
$\{ c \in E_C \mid d(a_i) c d(b_l) \in B_3(X_Z)\}$.
Hence we have 
$
A^G\hat{D} = \hat{D} B^G.
$
We then have that 
$\hat{D}^t(\id - (A^G)^t)\Z^{N_A} \subset (\id - (B^G)^t)\Z^{N_B}
$
so that $\hat{D}^t$ induces a desired homomorphism.
\end{proof}
The above homomorphism 
%induced by $\hat{D}^t$
from
$\Z^{N_A}/{(\id - {(A^G)}^{t})\Z^{N_A}}$
to 
$\Z^{N_B}/{(\id - {(B^G)}^{t})\Z^{N_B}}$
induced by  ${\hat{D}}^t$
is denoted by 
$\Phi_{\hat{D}^t}$.

Let us denote by 
$[e_i^{N_A}] $
 the class of the vector 
$e_i^{N_A} = (0,\dots,0,\overset{i}{1},0,\dots,0) \in \Z^{N_A}$
in $\Z^{N_A}/{(\id - {(A^G)}^{t})\Z^{N_A}}. 
$ 
%whose $i$th component is $1$, otherwise zero. 
It was shown in \cite{Cu3} that
the correspondence
$\epsilon_{A^G}: K_0({\mathcal{O}}_{A^G}) 
\rightarrow \Z^{N_A}/{(\id - {(A^G)}^{t})\Z^{N_A}}$
defined by
$\epsilon_{A^G}([S_{a_i}S_{a_i}^*] ) = [e_i^{N_A}]$
yields an isomorphism of abelian groups. 
We then have

%%%%%%%%%%%%%%%%%%%%%%%%%%%%%%%%
\begin{proposition}\label{prop:KTD}
Suppose that $A=CD, B=DC.$
Let  
$\Phi:\SOA \rightarrow \SOB$ 
be the isomorphism in Proposition \ref{prop:main1}.
Then the diagram 
$$
\begin{CD}
K_0({\mathcal{O}}_{A^G}) @>\Phi_* >> K_0({\mathcal{O}}_{B^G}) \\
@V{\epsilon_{A^G} }VV  @VV{\epsilon_{B^G}}V \\
\Z^{N_A}/{(\id - {(A^G)}^{t})\Z^{N_A}} @> \Phi_{\hat{D}^t}>> \Z^{N_B}/{(\id - {(B^G)}^{t})\Z^{N_B}}. 
\end{CD}
$$ 
is commutative.
\end{proposition}
\begin{proof}
We note that 
$\K = \K(\ell^2(\N))$ has a countable basis and $\N$ is decomposed such as
$\N = \cup_{j=1}^\infty \N_j$ where $\N_j$ is also disjoint infinite set 
such as $\N_j = \cup_{k=0}^\infty \N_{j_k}$ with disjoint infinite sets
$\N_{j_k}$ for every $k=0,1,2,\dots.$ 
We write $\N_{j_k}$ as
$
\N_{j_k} = \{ j_k(0), j_k(1), j_k(2), \dots \}.
$
In particular for $j=1, k=0$, 
we denote by
$\bar{n} = 1_0(n)$ for $n=0,1,2, \dots$
so  that 
$\N_{1_0} = \{ \bar{0}, \bar{1}, \bar{2}, \dots \}$.
Let
$p_{\bar{n}}, n=0,1,2,\dots
$
be the sequence of projections of rank one in $\K$  such that
$\sum_{n=0}^\infty p_{\bar{n}} = f_{1_0}.$
By \cite{Cu3}, the group 
$K_0(\OAG)$ is generated by the projections of the form
$$
S_{a_i}S_{a_i}^*\otimes p_{\bar{0}}, \qquad i=1,\dots,N_A.
$$
Denote by $1_A$ the unit of $\OAG$
so that 
$[1_A] = \sum_{i=1}^{N_A} [S_{a_i}S_{a_i}^*\otimes p_{\bar{0}}]$
in $K_0(\OAG)$.
Let $\Phi = \Ad(w):\OAG\otimes\K\rightarrow \OBG\otimes\K$ 
be the isomorphism
in Proposition \ref{prop:main1}.
Hence 
$\Phi_*: K_0(\OAG) \rightarrow K_0(\OBG)$ satisfies
$\Phi_*([S_{a_i}S_{a_i}^*\otimes p_{\bar{0}}])
 = [w(S_{a_i}S_{a_i}^*\otimes p_{\bar{0}})^*w^*].
$ 
To complete the proof of the proposition, 
we provide the following two lemmas.
\end{proof}
Let $l(i)$ be the number $l=1,\dots,N_C$ satisfying 
$c_l = c(a_i)$ so that $d(l(i)) \in E_D$
satisfies
$T_{l(i)} =S_{d(l(i))} S_{c(a_i)}S_{c(a_i)}^*$ in \eqref{eq:bl}.
We put
$s_{1_{l(i)},1_0} = s_{1_{l(i)},1}s_{1,1_0}$ 
and
$s_{1_0,1_{l(i)}}= s_{1_{l(i)},1_0}^*.
$
\begin{lemma} \label{lem:series}Keep the above notation.
\begin{enumerate}
\renewcommand{\theenumi}{\roman{enumi}}
\renewcommand{\labelenumi}{\textup{(\theenumi)}}
\item $w(S_{a_i}S_{a_i}^*\otimes p_{\bar{0}})w^* 
= v_B(S_{a_i}S_{a_i}^*\otimes s_{1,1_0}p_{\bar{0}} s_{1_0,1})v_B^*.
$
\item
$v_B(S_{a_i}S_{a_i}^*\otimes s_{1,1_0}p_{\bar{0}} s_{1_0,1})v_B^*
= S_{d(l(i))} S_{c(a_i)} S_{d(a_i)} S_{d(a_i)}^* S_{c(a_i)}^* S_{d(l(i))}^* \otimes 
s_{1_{l(i)},1_0}p_{\bar{0}} s_{1_0,1_{l(i)}}.
$ 
\end{enumerate}
\end{lemma}
\begin{proof}
(i)
The unitary $w$ is given by
$w = v_B v_A^*$.
We know 
$v_A = \sum_{n=1}^\infty v_n$
and 
$v_1 = P_C \otimes s_{1_0,1} + \sum_{k=1}^{N_D} U_k\otimes s_{1_k,1}.$
As
$p_{\bar{0}} s_{1_k,1} =0$ for $k=1,\dots, N_D$,
we have
\begin{align*}
v_A^*(S_{a_i}S_{a_i}^*\otimes p_{\bar{0}}) v_A 
& = v_1^* (S_{a_i}S_{a_i}^*\otimes p_{\bar{0}})v_1 \\
& = (P_C \otimes s_{1_0,1})^* (S_{a_i}S_{a_i}^*\otimes p_{\bar{0}}) (P_C \otimes s_{1_0,1}) \\ 
& = S_{a_i}S_{a_i}^*\otimes s_{1,1_0}p_{\bar{0}} s_{1_0,1}.
\end{align*}
(ii)
For $c_l \in E_C =\{c_1,\dots,c_{N_C}\}$ and $a_i\in E_A$,
we note that 
$S_{c_l}^*S_{a_i} = S_{c_l}^* S_{c(a_i)} S_{d(a_i)}$
if $c_l = c(a_i)$, otherwise zero.
Hence we have
\begin{align*}
  &v_B(S_{a_i}S_{a_i}^*\otimes s_{1,1_0}p_{\bar{0}} s_{1_0,1})v_B^* \\
= & (\sum_{l=1}^{N_C} T_l \otimes s_{1_l,1})
     (S_{a_i}S_{a_i}^*\otimes s_{1,1_0}p_{\bar{0}} s_{1_0,1}) 
     (\sum_{l'=1}^{N_C} T_{l'} \otimes s_{1_{l'},1})^* \\
= &\sum_{l=1}^{N_C} S_{d(l)} S_{c_l} S_{c_l}^* 
      S_{a_i}S_{a_i}^* S_{c_l} S_{c_l}^*S_{d(l)}^* 
     \otimes s_{1_l,1}s_{1,1_0}p_{\bar{0}} s_{1_0,1} s_{1_l,1}^* \\
=& S_{d(l(i))} S_{c(a_i)} S_{d(a_i)} S_{d(a_i)}^* S_{c(a_i)}^* S_{d(l(i))}^* 
     \otimes 
     s_{1_{l(i)},1_0}p_{\bar{0}} s_{1_0,1_{l(i)}}.
\end{align*}
\end{proof}
\begin{lemma}\label{lem:Sdai}
$ S_{d(a_i)}S_{d(a_i)}^* = \sum_{l=1}^{N_B} \hat{D}(i,l) S_{b_l}S_{b_l}^*.$ 
\end{lemma}
\begin{proof}
In the algebra $\OBG$, we have 
$\sum_{l=1}^{N_B} S_{b_l} S_{b_l}^* = 1$.
As $b_l = d(b_l)c(b_l)$, it implies that 
$
\sum_{l=1}^{N_B} S_{d(b_l)}S_{c(b_l)}S_{c(b_l)}^*S_{d(b_l)}^* = P_D
$
in
${\mathcal{O}}_Z$.
 By multiplying $S_{d(a_i)} S_{d(a_i)}^*$ to the equality we have
\begin{equation*}
\sum_{l=1}^{N_B} 
S_{d(a_i)} S_{d(a_i)}^*S_{d(b_l)}S_{c(b_l)}S_{c(b_l)}^*S_{d(b_l)}^*S_{d(a_i)} S_{d(a_i)}^*
 = S_{d(a_i)} S_{d(a_i)}^*.
\end{equation*}
Since
\begin{equation*}
S_{d(a_i)} S_{d(a_i)}^*S_{d(b_l)} = \hat{D}(i,l) S_{d(b_l)},
\end{equation*}
we have
\begin{equation*}
\sum_{l=1}^{N_B} 
\hat{D}(i,l) S_{d(b_l)} S_{c(b_l)}S_{c(b_l)}^*S_{d(b_l)}^*
 = S_{d(a_i)} S_{d(a_i)}^*.
\end{equation*}
As $S_{b_l} = S_{d(b_l)} S_{c(b_l)}$, 
we get the desired equality.
\end{proof}

{\it{Proof of Proposition \ref{prop:KTD}}}:

By using Lemma \ref{lem:series},
we have the equalities in $K_0(\OBG)$:
\begin{equation*}
 \Phi_*([S_{a_i}S_{a_i}^*\otimes p_{\bar{0}}]) 
%=&[w (S_{a_i}S_{a_i}^*\otimes p_{\bar{0}})^* w^*] \\
%=&[w( S_{a_i}S_{a_i}^*\otimes p_{\bar{0}})w^*] \\
%=&[ v_B(S_{a_i}S_{a_i}^*\otimes s_{1,1_0}p_{\bar{0}} s_{1_0,1})v_B^*]\\
=[ S_{d(l(i))} S_{c(a_i)} S_{d(a_i)} S_{d(a_i)}^* S_{c(a_i)}^* S_{d(l(i))}^* \otimes 
s_{1_{l(i)},1_0}p_{\bar{0}} s_{1_0,1_{l(i)}}].
\end{equation*}
Since
\begin{align*}
 & [ S_{d(l(i))} S_{c(a_i)} S_{d(a_i)} S_{d(a_i)}^* S_{c(a_i)}^* S_{d(l(i))}^* \otimes 
    s_{1_{l(i)},1_0}p_{\bar{0}} s_{1_0,1_{l(i)}}] \\
=& 
[ S_{d(a_i)} S_{d(a_i)}^*  \otimes 
f_{1_0} p_{\bar{0}} f_{1_0}]
\quad \text{ in } \quad K_0(\OBG),
\end{align*}
and
$f_{1_{0}} p_{\bar{0}} f_{1_0} \ge p_{\bar{0}}$,
we have
\begin{equation*}
\Phi_*([S_{a_i}S_{a_i}^*\otimes p_{\bar{0}}])  =
[ S_{d(a_i)} S_{d(a_i)}^*  \otimes p_{\bar{0}}].
\end{equation*}
As $\epsilon_{A^G}([S_{a_i}S_{a_i}^*\otimes p_{\bar{0}}]) = [e_i^{N_A}]$
and
$\epsilon_{B^G}([S_{b_l}S_{b_l}^*\otimes p_{\bar{0}}]) = [e_l^{N_B}],$
By using Lemma \ref{lem:Sdai}, we complete the proof of Proposition \ref{prop:KTD}.
\qed

\medskip

%%%%%%%%%%%%%%%%%%%%%%%%%%%%%%%%%%%%%%%%%%
Let $S_A$ and $R_A$ be the $N_A\times N$ matrix and $N\times N_A$ matrix 
 defined by 
\begin{equation*}
S_A(i,j) =
\begin{cases}
1 & \text{ if }  t(a_i) = v_j^A,\\
0 & \text{ otherwise, }
\end{cases}\qquad
R_A(j,i) =
\begin{cases}
1 & \text{ if }   v_j^A = s(a_i),\\
0 & \text{ otherwise, }
\end{cases}
\end{equation*}
for $i = 1,\dots, N_A$ and $j=1,\dots,N,$ respectively.
We then have
$A = R_A S_A$ and 
$A^G = S_A R_A$.
We similarly have the matrices 
$S_B, R_B$ for the other matrix $B$ such that 
$B = R_B S_B$ and 
$B^G = S_B R_B$.
The matrix $S_A^t: \Z^{N_A}\rightarrow \Z^N$ induces 
a homomorphism 
$\Z^{N_A}/{(\id - {(A^G)}^{t})\Z^{N_A}}
\rightarrow
\Z^{N}/{(\id - {A}^{t})\Z^{N}}$
of abelian groups 
which is actually an isomorphism
since its inverse is given by a homomorphism induced by
$R_A^t$.
The above isomorphism is denoted by $\Phi_{S_A^t}$.
We  have an isomorphism
$\Phi_{S_B^t}:\Z^{N_B}/{(\id - {(B^G)}^{t})\Z^{N_B}} 
\rightarrow
\Z^{M}/{(\id - {B}^t)\Z^{M}}
$
in a similar way.

Now we are assuming that
$A=CD, B=DC$ so that  
$AC =CB$ and hence $C^t A^t = B^t C^t$.  
The matrix $C^t: \Z^{N}\rightarrow \Z^{M}$ 
induces a homomorphism from
$\Z^{N}/{(\id - {A}^{t})\Z^{N}}$
to 
$\Z^{M}/{(\id - {B}^{t})\Z^{M}}$
as abelian groups, which is denoted by
$\Phi_{C^t}$. 
It is actually an isomorphism with $\Phi_{D^t}$ as its inverse.
We notice the following lemma. 
The second assertion (ii) is pointed out by Hiroki Matui.
The author thanks him for his advice. 
\begin{lemma}\label{lem:5.5}
\begin{enumerate}
\renewcommand{\theenumi}{\roman{enumi}}
\renewcommand{\labelenumi}{\textup{(\theenumi)}}
\item
The diagram
$$
\begin{CD}
\Z^{N_A}/{(\id - {(A^G)}^{t})\Z^{N_A}} @> \Phi_{\hat{D}^t}>>
 \Z^{N_B}/{(\id - {(B^G)}^{t})\Z^{N_B}} \\
@V{\Phi_{S^t_{A}} }VV  @VV{\Phi_{S^t_{B}} }V \\
\Z^{N}/{(\id - {A}^{t})\Z^{N}} @> \Phi_{C^t}>> \Z^{M}/{(\id - {B}^t)\Z^{M}} \\
\end{CD}
$$
 is commutative.  
\item
$\Phi_{S^t_A}([(1,1,\dots,1)]) = [(1,1,\dots,1)].$
\end{enumerate}
\end{lemma}
%The two vertical arrows above induced by the matrices $S_A^t$ and $S_B^t$
%are both isomorphisms since they have inverses for each 
%induced by the matrices $R_A^t$ and $R_B^t$, respectively.
\begin{proof}
(i)
Since $\Phi_{\hat{D}}$ is induced by the matrix $\hat{D}^t$,
it suffices to prove the equality
$\hat{D} S_B = S_A C$.
Let $(i,j)$ be  $i=1,\dots,N_A$ and $j= 1,\dots, M$
so that $a_i \in E_A$ and $v_j^B \in V_B$.
Let $k$ be such that $t(a_i) = v_k^A$.
Hence we have
\begin{equation*}
[S_A C](i,j) = \sum_{n=1}^NS_A(i,n) C(n,j) = C(k,j)
\end{equation*}
which is 
the number of edges of $E_C$ leaving $v_k^A$ and terminating at $v_j^B$. 
On the other hand, 
\begin{equation*}
[\hat{D} S_B](i,j) = \sum_{l=1}^{M_B} \hat{D}(i,l) S_B(l,j).
\end{equation*}
It is easy to see that the above number is also $C(k,j).$

(ii)
Since $A = R_A S_A$,
for each $k=1,\dots, N_A$ with $a_k \in E_A$
there exists a unique $i=1,\dots,N$ such that $s(a_k) = v_i^A$.
Hence
$\sum_{i=1}^N R_A(i,k) = 1$ so that we have
for each $j=1,\dots,N$
\begin{equation*}
\sum_{i=1}^N A^t(j,i)
 = \sum_{i=1}^N \sum_{k=1}^{N_A} R_A(i,k)S_A(k,j)
 = \sum_{k=1}^{N_A} (\sum_{i=1}^N  R_A(i,k)) S_A(k,j)
% =  \sum_{k=1}^{N_A} S_A(k,j)
=  \sum_{k=1}^{N_A} S_A^t(j,k).
\end{equation*}
We then see
\begin{align*}
\Phi_{S^t_A}([(1,1,\dots,1)]) 
& = [( \sum_{k=1}^{N_A} S_A(k,1), \sum_{k=1}^{N_A} S_A(k,2), \dots,
         \sum_{k=1}^{N_A} S_A(k,N))] \\
& = [( \sum_{i=1}^N A^t(1,i),  \sum_{i=1}^N A^t(2,i), \dots,
          \sum_{i=1}^N A^t(N,i))] \\
& =[(1,1,\dots,1)] \quad \text{ in } \Z^{N}/{(\id - {A}^{t})\Z^{N}}.
\end{align*}
\end{proof}
Put
$\epsilon_A = \Phi_{S_A^t} \circ \epsilon_{A^G}: 
K_0(\OA) \rightarrow  \Z^{N}/{(\id - {A}^{t})\Z^{N}}$,
which is an isomorphism of groups such that 
$\epsilon_A([1_A]) =[(1,1,\dots,1)].$ 
We thus reach the following theorem:
\begin{theorem}\label{thm:KC}
Suppose that two nonnegative irreducible matrices $A, B$ satisfy 
$A= CD, B= DC$ for some nonnegative rectangular matrices $C, D$.
Then the diagram
$$
\begin{CD}
K_0({\OA}) @>\Phi_* >> K_0({\OB}) \\
@V{\epsilon_{A} }VV  @VV{\epsilon_{B}}V \\
\Z^{N}/{(\id - {A}^{t})\Z^{N}} @> \Phi_{C^t} >> \Z^{M}/{(\id - {B}^t)\Z^{M}} 
\end{CD}
$$
 is commutative,
where the two vertical arrows and the two horizontal arrows are all isomorphisms of abelian groups.
\end{theorem}

We write 
$
A \underset{C,D}{\approx}B
$
if 
$A = CD, \, B= DC$. 
Recall that $A, B$ are said to be strong shift equivalent in $n$-step
if there exist a finite sequence of square matrices 
$ A_1, \dots, A_{n-1}$
and two finite sequences of rectangular matrices
$C_1, \dots, C_n$ and 
$D_1, \dots,D_n$ such that
\begin{equation*}
A =A_0 \underset{C_1,D_1}{\approx}A_1, \quad 
A _1\underset{C_2,D_2}{\approx}A_2, \quad \dots, \quad
A_{n-1} \underset{C_{n},D_{n}}{\approx}A_n =B.
\end{equation*}
This situation is written
\begin{equation}
A \underset{C_1,D_1}{\approx}\cdots \underset{C_{n},D_{n}}{\approx}B. \label{eq:5.2}
\end{equation}
R. F. Williams proved that two-sided topological Markov shifts 
$(\bar{X}_A,\bar{\sigma}_A)$ and 
$(\bar{X}_B,\bar{\sigma}_B)$
are topologically conjugate if and only if 
$A$ and $B$ are strong shift equivalent in $n$-step for some $n$
(\cite{Williams}).
%Let us denote  the group $\Z^{N}/{(\id - {A}^{t})\Z^{N}}$ by $K_0(A)$.
Hence we have the following corollary.
\begin{corollary}\label{cor:KC}
Suppose that
two matrices $A, B$ are strong shift equivalent in $n$-step
%\begin{equation}
%A \underset{C_1,D_1}{\approx}\cdots \underset{C_{n},D_{n}}{\approx}B. \label{eq:SSEAB}
%\end{equation}
for some  two sequences of rectangular matrices
$C_1, \dots, C_n$ and 
$D_1, \dots,D_n$ as in \eqref{eq:5.2}.
Then there exist an isomorphism
$\Phi:\SOA \rightarrow \SOB $ of $C^*$-algebras 
and a unitary representation 
$t \in \T \rightarrow u_t^A \in M(\SDA)$ 
such that
\begin{equation*}
\Phi(\SDA) = \SDB, \qquad
\Phi \circ \Ad(u_t^A) \circ  (\rho^{A}_t \otimes\id) = (\rho^{B}_t \otimes \id) \circ \Phi,
\end{equation*}
and the following diagram is commutative
$$
\begin{CD}
K_0({\OA}) @>\Phi_* >> K_0({\OB}) \\
@V{\epsilon_{A} }VV  @VV{\epsilon_{B}}V \\
\Z^{N}/{(\id - {A}^{t})\Z^{N}} 
@> \Phi_{{(C_1 C_2 \cdots C_n)}^t}>> \Z^{M}/{(\id - {B}^t)\Z^{M}}. 
\end{CD}
$$
\end{corollary}
We note that the inverse 
of $\Phi_{{(C_1 C_2 \cdots C_n)}^t}:
\Z^{N}/{(\id - {A}^{t})\Z^{N}} \rightarrow  \Z^{M}/{(\id - {B}^t)\Z^{M}}$
is given by 
 $\Phi_{{(D_n \cdots D_2 D_1)}^t}:
 \Z^{M}/{(\id - {B}^t)\Z^{M}}\rightarrow 
\Z^{N}/{(\id - {A}^{t})\Z^{N}}.
$

%\begin{lemma}
%Suppose that $A=CD, B=DC.$
%The homomorphism  $ [x_i]_{i=1}^N \in \Z^N \rightarrow [\sum_{i=1}^N C(i,j)x_i]_{j=1}^N$
%of the abelian groups induces an isomorphism written 
%$\Z^N/(\id - A^t)\Z^N \rightarrow \Z^M/(\id - B^t)\Z^M$ of the abelian groups. 
%It is written  $\Phi_C$ and has 
%$\Phi_D:\Z^M/(\id - B^t)\Z^M\rightarrow \Z^N/(\id - A^t)\Z^N$ as its inverse.
%\end{lemma}

%%%%%%%%%%%%%%%%%%%%%%%%%%%%%%%%%%%%
%%%%%%%%%%%%%%%%%%%%%%%%%%%%%%%%%%%%
\section{Converse and Invariant}
%%%%%%%%%%%%%%%%%%%%%%%%%%%%%
In this section, we will study the converse of Corollary \ref{cor:main}
by using Corollary \ref{cor:KC}.
We fix a projection  $p_1$  of  rank one in $\K.$
\begin{proposition}
The following assertions are equivalent.
\begin{enumerate}
\renewcommand{\theenumi}{\roman{enumi}}
\renewcommand{\labelenumi}{\textup{(\theenumi)}}
\item There exist an isomorphism 
$\Phi:\SOA \rightarrow \SOB $
of $C^*$-algebras and
a unitary one-cocycle $u_t \in M(\SOB), t \in \T$ 
relative to  $\rho_t^B\otimes \id$ 
such that
\begin{gather}
\Phi(\SDA) = \SDB, \qquad
\Phi \circ (\rho^{A}_t \otimes\id) = \Ad(u_t) \circ (\rho^{B}_t \otimes \id) \circ \Phi, 
\label{eq:SSCOE1}\\
\Phi_*([1_A \otimes p_1]) = [1_B \otimes p_1] \text{ in } K_0(\OB). \label{eq:SSCOE2}
\end{gather}
\item There  exist
 an isomorphism
$\phi:\OA \rightarrow \OB$
and 
a unitary one-cocycle $v_t \in U(\OB), t \in \T$
relative to $\rho_t^B$ on $\OB$
such that 
\begin{equation}
\phi(\DA) = \DB
\quad 
\text{  and }
\quad
\phi \circ \rho^A_t = \Ad(v_t)\circ \rho^B_t \circ \phi, 
\qquad t \in {\mathbb{T}}. \label{eq:SCOE}
\end{equation}
\end{enumerate}
\end{proposition}
\begin{proof}
The implication (ii) $\Longrightarrow$ (i) is obvious 
by putting $\Phi = \varphi \otimes \id$ and $u_t = v_t \otimes 1$.
We will show the implication
(i) $\Longrightarrow$ (ii) in the following way.
By \cite[Proposition 3.13]{MaPAMS}, 
the condition
$
\Phi_*([1_A \otimes p_1]) = [1_B \otimes p_1] \text{ in } K_0(\OB)
$
ensures us that there exists a partial isometry
$V \in \SOB$ satisfying the following conditions:
\begin{gather*}
V(\SDB)V^* \subset \SDB, \qquad
V^*(\SDB)V \subset \SDB, \\
VV^* = 1_B\otimes p_1,\quad V^*V = \Phi(1_A\otimes p_1).
\end{gather*}
 Put
$\Psi = \Ad(V) \circ \Phi: \SOA \rightarrow \SOB.$
It is straightforward to see that 
\begin{equation*}
\Psi(\OA\otimes \C p_1) = \OB\otimes \C p_1,
\qquad 
\Psi(\DA\otimes \C p_1) = \DB\otimes \C p_1,
\qquad
\Psi(1_A\otimes p_1) = 1_B\otimes p_1.
 \end{equation*}
It is clear that 
$\Psi_* = \Phi_*: K_0(\OA) \rightarrow K_0(\OB)$.
We identify $\OB\otimes\C p_1$ with $\OB$.
Put
the partial isometry 
$v_t = V u_t (\rho^B_t\otimes \id)(V^*) \in  \SOB$.
Since
$v_t = (1_B\otimes p_1) v_t (1_B\otimes p_1) $,
by this identification,
 $v_t $ belongs to $\OB$.
Define 
$\phi: \OA \rightarrow \OB$ 
by setting
$\phi(a) = \Psi(a\otimes p_1)$ for $a \in \OA$.
 It then follows that
\begin{align*}
\phi(\rho^A_t(a))\otimes p_1
=& V \Phi( \rho^A_t(a) \otimes p_1) V^* \\
=& V (\Ad(u_t)\circ (\rho^B_t\otimes \id) \circ \Phi)( a \otimes p_1) V^* \\
=& V u_t (\rho^B_t\otimes \id)(V^*) (\rho^B_t\otimes \id) \Phi(V( a \otimes p_1) V^*)
     (\rho^B_t\otimes \id)(V)u_t^* V^* \\
= & v_t ((\rho^B_t\otimes \id)\circ\Psi) (a\otimes p_1) v_t^* \\
%= & v_t (\rho^B_t\otimes \id)(\phi(a)\otimes p_1) v_t^* \\
= & (\Ad(v_t) \circ (\rho^B_t \circ \phi)(a)) \otimes p_1
\end{align*}
so that we have 
$\phi(\rho^A_t(a)) = (\Ad(v_t)\circ\rho^B_t\circ\phi)(a).$
Since we have
$$
(\rho^B_t\otimes \id)(\Phi(1_A\otimes p_1))
= (\Ad(u_t^*)\circ \Phi \circ (\rho^A_t \otimes\id))(1_A\otimes p_1)
=u_t^* \Phi(1_A \otimes p_1)u_t
=u_t^* V^*V u_t,
$$
we have
\begin{align*}
v_t \rho^B_t(v_s) 
= & V u_t (\rho^B_t\otimes \id)(V^*) (\rho^B_t\otimes \id)(V u_s(\rho^B_s\otimes \id)(V^*)) \\
= & V u_t (\rho^B_t\otimes \id)(V^* V) (\rho^B_t\otimes \id)(u_s)
      (\rho^B_t\circ \rho^B_s\otimes \id)(V^*) \\
= & V u_t (\rho^B_t\otimes \id)(\Phi(1_A\otimes p_1)) (\rho^B_t\otimes \id)(u_s)
      (\rho^B_{t+s}\otimes \id)(V^*) \\
= & V u_t u_t^* V^* V u_t
 (\rho^B_t\otimes \id)(u_s)
 (\rho^B_{t+s}\otimes \id)(V^*) \\ 
= & V u_t (\rho^B_t\otimes \id)(u_s) (\rho^B_{t+s}\otimes \id)(V^*) \\ 
= & V u_{t+s} (\rho^B_{t+s}\otimes \id)(V^*) \\ 
= & v_{t+s}. 
\end{align*}
Hence $v_t, t \in \T$ is a unitary one-cocycle relative to $\rho^B\otimes \id.$
\end{proof}
\begin{remark}
Let $v_t$ in $\OB$ 
be
a unitary one-cocycle  relative to $\rho^B_t$ satisfying \eqref{eq:SCOE}.
For $a \in \DA$, we see that 
$\varphi(\rho^A_t(a)) = \Ad(v_t)(\rho^B_t(\varphi(a))).$
As  $\rho^A_t(a) = a$ and $\varphi(a)$ belongs to $\DB$ so that we have
$\varphi(a) =\Ad(v_t)(\varphi(a))$.
Hence $v_t$ commutes with any element of $\DB$.
This implies that $v_t$ belongs to $\DB$ and hence 
it is fixed by the action $\rho^B$.
Therefore 
a unitary one-cocycle
$v_t$ in $\OB$
relative to $\rho^B_t$ satisfying \eqref{eq:SCOE}
automatically belongs to $\DB$ and yields a unitary representation
$t \in \T \rightarrow v_t \in \DB$.
Since the unitary $u_t$ in \eqref{eq:SSCOE1} is given by
$u_t = v_t \otimes 1$ from the unitary $v_t$ satisfying  \eqref{eq:SCOE},
the unitary one-cocycle $u_t$ in the statement (i)  of the above proposition
can be taken as a unitary representation 
$t \in \T \rightarrow u_t \in M(\SDB)$ which is fixed by the action
$\rho_t^B\otimes\id$.
\end{remark}

\begin{corollary}\label{cor:SOTWO}
If there exist an isomorphism
$\Phi:\SOA \rightarrow \SOB $  of $C^*$-algebras 
and a unitary one-cocycle $u_t$ in $M(\SOB)$ relative to $\rho^B_t \otimes \id$ such that
\begin{gather*}
\Phi(\SDA) = \SDB, \qquad
\Phi \circ (\rho^{A}_t \otimes\id) = \Ad(u_t)\circ (\rho^{B}_t \otimes \id) \circ \Phi, \\
\Phi_*([1_A \otimes p_1]) = [1_B \otimes p_1] \text{ in } K_0(\OB),
\end{gather*}
then two-sided topological Markov shifts
$(\bar{X}_A, \bar{\sigma}_A)$ and
$(\bar{X}_B, \bar{\sigma}_B)$
are topologically conjugate.
\end{corollary}
\begin{proof}
By \cite[Theorem 6.7]{MaJOT2015}, the equality \eqref{eq:SCOE}
implies strongly continuous orbit equivalence between the one-sided topological 
Markov shifts
$(X_A, \sigma_A)$ and $(X_B,\sigma_B)$.
It also implies topological conjugacy of their two-sided topological Markov shifts
$({\bar{X}}_A, {\bar{\sigma}}_A)$
and
$({\bar{X}}_B, {\bar{\sigma}}_B)$ by \cite[Theorem 5.5]{MaJOT2015}.
\end{proof}
\begin{definition}
We say that an isomorphism 
$ \xi:\SOB \rightarrow \SOA$
of $C^*$-algebras is induced from strong shift equivalence
$A \underset{C_1,D_1}{\approx}\cdots \underset{C_{n},D_{n}}{\approx}B$
if there exists a unitary one-cocycle $u_t$ in $M(\SOA)$ relative to $\rho^A_t \otimes \id$
such that  
\begin{gather*}
 \xi(\SDB) = \SDA, \qquad
\xi \circ (\rho^{B}_t \otimes\id) = \Ad(u_t) \circ (\rho^{A}_t \otimes \id) \circ \xi, \\
\xi_* =\epsilon_{A}^{-1}\circ \Phi_{(D_n \cdots D_2 D_1)^t}\circ\epsilon_{B} 
: K_0(\OB) \rightarrow K_0(\OA).
\end{gather*}
\end{definition}
We will define the strong shift equivalence invariant subset of $K_0(\OA)$
as follows.
\begin{definition}
\begin{align*}
\operatorname{K}_0^{\text{SSE}}(\OA)
=\{&  [p] \in K_0(\OA) \mid
\exists B \text{ a square matrix  and }  
 \exists \xi:\SOB \rightarrow \SOA \\
& \text{ an isomorphism induced from strong shift equivalence}; \\
& A \underset{C_1,D_1}{\approx}\cdots \underset{C_{n},D_{n}}{\approx}B  \text{ and }  
 \xi_*([1_B ]) =[p]
\text{ in } K_0(\OA) \}.
\end{align*}
\end{definition}
We note that 
the class $[1_A]$ in $K_0(\OA)$ of the unit $1_A$ of $\OA$ always belongs to the set
$\operatorname{K}_0^{\text{SSE}}(\OA)$,
because we may take $B=A$ and $\xi = \id$. 
\begin{proposition}\label{prop:K0SSEOA}
Suppose that there exists a topological conjugacy between
$(\bar{X}_A, \bar{\sigma}_A)$ and
$(\bar{X}_B, \bar{\sigma}_B)$.
Then there exists an isomorphism
$\eta: K_0(\OA) \rightarrow K_0(\OB)$
satisfying
$\eta(\operatorname{K}_0^{\text{SSE}}(\OA)) 
=\operatorname{K}_0^{\text{SSE}}(\OB).
$
Hence the pair
$(K_0(\OA), \operatorname{K}_0^{\text{SSE}}(\OA))$
is an invariant under topological conjugacy of two-sided topological Markov shifts.
\end{proposition}
\begin{proof}
Suppose that 
$(\bar{X}_A, \bar{\sigma}_A)$ and
$(\bar{X}_B, \bar{\sigma}_B)$
are topologically conjugate so that
$
A \underset{C_1,D_1}{\approx}\cdots \underset{C_{n},D_{n}}{\approx}B
$ 
for some nonnegative rectangular matrices
$C_1, D_1, \dots, C_n, D_n$.
The strong shift equivalence
induces that there exist an isomorphism
$\xi_{BA}:\SOA \rightarrow \SOB$
and a unitary one-cocycle $u_t$ in $M(\SOB)$ relative to $\rho^B_t \otimes \id$
 such that
\begin{gather*}
\xi_{BA}(\SDA) = \SDB, \qquad
\xi_{BA} \circ (\rho^{A}_t \otimes\id) = \Ad(u_t) \circ (\rho^{B}_t \otimes \id) \circ \xi_{BA}, \\
\xi_{BA *} = \Phi_{(C_1 \cdots C_n)^t}: K_0(\OA) \rightarrow  K_0(\OB).
\end{gather*}
%Hence there exists an isomorphism of $C^*$-algebras
%$\Phi:\SOA \rightarrow \SOB $ such that
%\begin{equation*}
%\Phi(\SDA) = \SDB, \qquad
%\Phi \circ (\rho^{A}_t \otimes\id) = (\rho^{B}_t \otimes \id) \circ \Phi. \label{eq:isom}
%\end{equation*}
Put 
$\eta=\xi_{BA*}: K_0(\OA) \rightarrow K_0(\OB).$
Take an element 
$[p] \in \operatorname{K}_0^{\text{SSE}}(\OA)$.
There exists 
 a square nonnegative matrix $A'$
and
an isomorphism $\xi_{AA'}:{\mathcal{O}}_{A'} \otimes \K \rightarrow \SOA$
of $C^*$-algebras induced from strong shift equivalence
$ A' \underset{C'_1,D'_1}{\approx}\cdots \underset{C'_{n'},D'_{n'}}{\approx}A$
such that   
$ \xi_{AA'*}([1_{A'} ]) =[p]$ in $K_0(\OA)$.
Then the isomorphism
$\xi_{BA} \circ \xi_{AA'}:{\mathcal{O}}_{A'} \otimes \K \rightarrow \SOB$
is induced from  strong shift equivalence
$$ 
A' \underset{C'_1,D'_1}{\approx}\cdots \underset{C'_{n'},D'_{n'}}{\approx}A=
A \underset{C_1,D_1}{\approx}\cdots \underset{C_{n},D_{n}}{\approx}B
$$
such that   
$ \eta([p]) = (\xi_{BA}\circ \xi_{AA'})_*([1_{A'} ])$ in $K_0(\OB)$
so that 
$ \eta([p]) \in \operatorname{K}_0^{\text{SSE}}(\OB)$.
\end{proof}

%%%%%%%%%%%%%%%%%%%%%%%%%%%%%%%%%%%%%%%%
Suppose that 
two matrices $A, B$ are strong shift equivalent in $n$-step
such as \eqref{eq:5.2}.
%\begin{equation}
%A \underset{C_1,D_1}{\approx}\cdots \underset{C_{n},D_{n}}{\approx}B. \label{eq:SSEAB}
%\end{equation} 
The matrix $B$ in \eqref{eq:5.2} 
is given by 
$B = D_n C_n$ so that \eqref{eq:5.2} is written as 
\begin{equation}
A \underset{C_1,D_1}{\approx}\cdots \underset{C_{n},D_{n}}{\approx}D_n C_n. \label{eq:SSEADC}
\end{equation}
We set the following sequence  $\text{SSE}_n(A), n=1,2,\dots$ of subsets of the group 
$\Z^N$ 
\begin{equation*}
%\text{SSE}_1(A)
%& = \{ v \in \Z^N \mid v = [1,\dots,1] D, \ A \underset{C,D}{\approx}DC \},   \\
\text{SSE}_n(A)
 = \{ v \in \Z^N \mid v = D_1^t \cdots D_{n-1}^t D_n^t [1,1,\dots,1]^t, \
A \underset{C_1,D_1}{\approx}\cdots \underset{C_{n},D_{n}}{\approx}D_n C_n \},  
\end{equation*}
where $[1,1,\dots,1]^t$ denotes the $ (\text{the row size of } D_n)\times 1$ matrix
whose entries are all $1$'s.
We  define the sequence  
$\operatorname{K}_{\text{alg},n}^{\text{SSE}}(A), n=1,2,\dots$ 
of subsets of the group 
$\Z^N/(\id - A^t)\Z^N$ by
\begin{equation*}
%\operatorname{K}_{\text{alg},1}^{\text{SSE}}(A)
%& = \{ [v] \in \Z^N/(\id - A^t)\Z^N \mid  v \in \text{SSE}_1(A) \}, \\
\operatorname{K}_{\text{alg},n}^{\text{SSE}}(A)
 = \{ [v] \in \Z^N/(\id - A^t)\Z^N \mid  v \in \text{SSE}_n(A) \}, \quad n=1,2,\dots.
\end{equation*}
Then we define the subset 
$\operatorname{K}_{\text{alg}}^{\text{SSE}}(A) $
of $\Z^N/(\id - A^t)\Z^N$ by
\begin{equation*}
\operatorname{K}_{\text{alg}}^{\text{SSE}}(A) = \cup_{n=1}^\infty
\operatorname{K}_{\text{alg},n}^{\text{SSE}}(A).
\end{equation*}
By Corollary \ref{cor:KC}, we have the following proposition
\begin{proposition}\label{prop:algSSE}
Let $\epsilon_A: K_0(\OA) \rightarrow  \Z^N/(\id - A^t)\Z^N $ be the isomorphism
in Corollary \ref{cor:KC}.
Then we have
\begin{equation*}
\epsilon_A(K_0^{\text{SSE}}(\OA)) =
\operatorname{K}_{\operatorname{alg}}^{\text{SSE}}(A).
\end{equation*}
\end{proposition}
\begin{proof}
For $[p] \in K_0^{\text{SSE}}(\OA)$,
there exist a nonnegative square matrix $B$ 
with a strong shift equivalence 
$A \underset{C_1,D_1}{\approx}\cdots \underset{C_{n},D_{n}}{\approx}B$
and
 an isomorphism
$\xi:\SOB \rightarrow \SOA $ of $C^*$-algebras 
and a unitary one-cocycle $u_t, t\in \T$ relative to
$\rho^A\otimes\id$ such that
\begin{gather}
\xi(\SDB) = \SDA, \qquad
\xi \circ (\rho^{B}_t \otimes \id) =\Ad(u_t) \circ  (\rho^{A}_t \otimes\id)\circ\xi, 
\label{eq:xi5.71}\\
\xi_* = \epsilon_A^{-1}\circ\Phi_{{(D_n  \cdots D_2 D_1)}^t}\circ\epsilon_B:
K_0(\OB) \rightarrow K_0(\OA)
\quad \text{ and } \quad \xi_*([1_B]) = [p]. 
\end{gather}
Since $\epsilon_B([1_B]) = [[1,1,\dots,1]^t]$ in $\Z^M/(\id - B^t)\Z^M$,
we have
\begin{equation}
\epsilon_A([p]) 
= \epsilon_A\circ \xi_*([1_B]) 
=\Phi_{{(D_n  \cdots D_2 D_1)}^t}\circ\epsilon_B([1_B])
=\Phi_{{(D_n  \cdots D_2 D_1)}^t}([1,1,\dots,1]^t) \label{eq:5.73}
\end{equation} 
so that 
$\epsilon_A([p]) \in \operatorname{K}_{\operatorname{alg}}^{\text{SSE}}(A)
$
and hence
$\epsilon_A(K_0^{\text{SSE}}(\OA)) \subset
\operatorname{K}_{\operatorname{alg}}^{\text{SSE}}(A).
$

Conversely,
take an arbitrary element 
$[v] \in \operatorname{K}_{\operatorname{alg}}^{\text{SSE}}(A).$
We may find  
a strong shift equivalence 
$A \underset{C_1,D_1}{\approx}\cdots \underset{C_{n},D_{n}}{\approx}D_n C_n$
such that 
$v = {(D_n  \cdots D_2 D_1)}^t[1,1,\dots,1]^t.$
Put $B = D_n C_n$.
By Corollary \ref{cor:KC},
there exists
 an isomorphism
$\xi:\SOB \rightarrow \SOA $ of $C^*$-algebras 
and a unitary one-cocycle $u_t, t\in \T$ relative to
$\rho^A\otimes\id$ 
satisfying 
\eqref{eq:xi5.71}
and
$\xi_* = \epsilon_A^{-1}\circ\Phi_{{(D_n  \cdots D_2 D_1)}^t}\circ\epsilon_B:
K_0(\OB) \rightarrow K_0(\OA)$.
Put
$[p] =\xi_*([1_B])$
which belongs to $K_0^{\text{SSE}}(\OA).$
By the same equalities as \eqref{eq:5.73},
we get 
$\epsilon_A([p]) 
=\Phi_{{(D_n  \cdots D_2 D_1)}^t}([1,1,\dots,1]^t)
$ which is the class of $[v]$.
This shows that 
$\epsilon_A(K_0^{\text{SSE}}(\OA)) \supset
\operatorname{K}_{\operatorname{alg}}^{\text{SSE}}(A).$
\end{proof}
\begin{theorem} \label{thm:main2}
Let $A, B$ be nonnegative irreducible and non-permutation matrices.   
The following two assertions are equivalent.
\begin{enumerate}
\renewcommand{\theenumi}{\roman{enumi}}
\renewcommand{\labelenumi}{\textup{(\theenumi)}}
\item Two-sided topological Markov shifts
$(\bar{X}_A, \bar{\sigma}_A)$ and
$(\bar{X}_B, \bar{\sigma}_B)$
are topologically conjugate.
\item 
There exist an isomorphism 
$\Phi:\SOA \rightarrow \SOB $ of $C^*$-algebras  
and
 a unitary one-cocycle $u_t$ in $M(\SOB)$ relative to $\rho^B_t \otimes \id$
such that
\begin{gather*}
\Phi(\SDA) = \SDB, \qquad
\Phi \circ (\rho^{A}_t \otimes\id) = \Ad(u_t) \circ (\rho^{B}_t \otimes \id) \circ \Phi, \\
\Phi_*(K_0^{\text{SSE}}(\OA)) = K_0^{\text{SSE}}(\OB) \text{ in } K_0(\OB).
\end{gather*}
\end{enumerate}
\end{theorem}
\begin{proof}
(i) $\Longrightarrow$ (ii) comes from 
Corollary \ref{cor:main} and Proposition \ref{prop:K0SSEOA}.

(ii) $\Longrightarrow$ (i):
Suppose that there exist an isomorphism 
$\Phi:\SOA \rightarrow \SOB$ of $C^*$-algebras
and  a unitary one-cocycle $u_t$ in $M(\SOB)$ relative to $\rho^B_t \otimes \id$
satisfying the conditions of (ii).
Put the projection $p = \Phi(1_A\otimes p_1) \in \SOB$.
As $[1_A] \in K_0^{\text{SSE}}(\OA)$ and
$\Phi_*(K_0^{\text{SSE}}(\OA)) = K_0^{\text{SSE}}(\OB)$,
the class  $[p] = \Phi_*([1_A])$ of $p $ in $K_0(\OB)$
belongs to
$K_0^{\text{SSE}}(\OB)$.
One may take a nonnegative square matrix $B'$ 
and an isomorphism
$\gamma: \SOB \rightarrow {\mathcal{O}}_{B'}\otimes \K$
with a unitary one-cocycle $u'_t$ in $M({\mathcal{O}}_{B'}\otimes \K)$
relative to $\rho^{B'}_t \otimes \id$ induced from
strong shift equivalence
$B \underset{C_1,D_1}{\approx}\cdots \underset{C_{n},D_{n}}{\approx}B'$
satisfying
\begin{gather*}
\gamma(\SDB)= {\mathcal{D}}_{B'}\otimes \C,
\qquad
\gamma \circ (\rho^{B}_t \otimes\id) 
=\Ad(u'_t)\circ (\rho^{B'}_t \otimes \id) \circ \gamma, \\
\gamma_*([p]) = [1_{B'}] \quad \text{ in } K_0({\mathcal{O}}_{B'}).
\end{gather*}
Then the isomorphism
$\gamma\circ \Phi: \SOA \rightarrow {\mathcal{O}}_{B'}\otimes \K$
satisfies the conditions
\begin{gather*}
(\gamma\circ \Phi)(\SDA)= {\mathcal{D}}_{B'}\otimes \C,
\qquad
(\gamma\circ \Phi) \circ (\rho^{A}_t \otimes\id) 
= \Ad(\gamma(u_t) u'_t)\circ  (\rho^{B'}_t \otimes \id) \circ (\gamma\circ \Phi), \\
(\gamma\circ \Phi)_*([1_A]) = [1_{B'}] \quad \text{ in } K_0({\mathcal{O}}_{B'}).
\end{gather*}
By Corollary \ref{cor:SOTWO}, the two-sided  topological Markov shifts
$ (\bar{X}_A, \bar{\sigma}_A) $
and
$ (\bar{X}_{B'}, \bar{\sigma}_{B'})$
are topologically conjugate.
Since
$ (\bar{X}_B, \bar{\sigma}_B) $
and
$ (\bar{X}_{B'}, \bar{\sigma}_{B'})$
are topologically conjugate, 
 so are   
$ (\bar{X}_A, \bar{\sigma}_A) $
and
$ (\bar{X}_{B}, \bar{\sigma}_{B})$.
\end{proof}
\begin{remark}
The unitary  one-cocycle $u_t$ in $M(\SOB)$ in (ii) of the above theorem
can be taken as a unitary representation $t \in \T \rightarrow u_t \in M(\SOB)$
by Corollary \ref{cor:main}.
\end{remark}
\begin{definition}
A nonnegative square matrix $A = [A(i,j)]_{i,j=1}^N$ 
is said to have {\it full strong shift equivalent units in } $K_0$-{\it group}
if  $\operatorname{K}_{\text{alg}}^{\text{SSE}}(A) = \Z^N / (\id - A^t)\Z^N$.
We simply call it  that $A$ \it{ has full units}.

\end{definition}
By Proposition \ref{prop:algSSE},
$A$ has full units if and only if 
$K_0^{\text{SSE}}(\OA) = K_0(\OA)$.
Since the subset $K_0^{\text{SSE}}(\OA) \subset K_0(\OA)$
is invariant under topological conjugacy of two-sided topological Markov shifts
by Proposition \ref{prop:K0SSEOA},
we have
%%%%%%%%%%%%%%%%%%%%%
\begin{proposition}
Suppose that two-sided topological Markov shifts
$(\bar{X}_A, \bar{\sigma}_A)$ and
$(\bar{X}_B, \bar{\sigma}_B)$
are topologically conjugate.
Then $A$ has full units if and only if 
$B$ has full units.
\end{proposition}
As a corollary of Theorem \ref{thm:main2}, we have the following corollary.
\begin{corollary}\label{cor:main3}
Suppose that both $A$ and $B$  have full units.
Then the following two assertions are equivalent.
\begin{enumerate}
\renewcommand{\theenumi}{\roman{enumi}}
\renewcommand{\labelenumi}{\textup{(\theenumi)}}
\item Two-sided topological Markov shifts
$(\bar{X}_A, \bar{\sigma}_A)$ and
$(\bar{X}_B, \bar{\sigma}_B)$
are topologically conjugate.
\item 
There exist an isomorphism $\Phi:\SOA \rightarrow \SOB $ 
of $C^*$-algebras
and  a unitary one-cocycle $u_t$ in $M(\SOB)$ relative to $\rho^B_t \otimes \id$
 such that
\begin{equation*}
\Phi(\SDA) = \SDB, \qquad
\Phi \circ (\rho^{A}_t \otimes\id) = \Ad(u_t) \circ (\rho^{B}_t \otimes \id) \circ \Phi. \\
\end{equation*}
\end{enumerate}
\end{corollary}
\begin{example}
\hspace{6cm}

{\bf 1.}
If $K_0(\OA) =0$, then $A$ has full units. 

{\bf 2.}
 If $A = [N]$ for some $1<N \in \N$, then the matrix $A$ has full units. 
For any $0\le k \le N-1$,
let
$C$ be the $1\times (k+1)$ matrix 
$[1,\dots,1, N-k]$ and
$D$  the $(k+1)\times 1$matrix
$(1,1,\dots,1)^t.
$
Then $A = CD$ and 
$D^t [1,\dots,1]^t = k+1$.
Hence
$[k+1]\in \Z/(1 - N)\Z$ so that 
$\operatorname{K}_{\text{alg}}^{\text{SSE}}(A) 
=\Z/(1 - N)\Z = K_0(\OA)$.
\end{example}

There is no known example of irreducible, non permutation matrix $A$
such that $A$ does not have full units. 
%The author  conjectures that all irreducible, non permutation matrices have 
%full units, so that Corollary \ref{cor:main3} would hold without assumption
%that both $A$ and $B$  have full units. 

%Some results of this paper will be generalize to subshifts in a forthcoming paper
%\cite{MaPre20162}.

\medskip

%%%%%%%%%%%%%%%%%%%%%%%%%%%%%%%%%%%%
%%%%%%%%%%%%%%%%%%%%%%%%%%%%%%%%%%%%%%
{\it Acknowledgments:}
%The author would like to thank 
%Hiroki Matui for his discussions on this subject.
This work was done during staying at the Institut Mittag-Leffler, and 
the author is grateful to the Institut Mittag-Leffler 
for its wonderful research environment, 
hospitality and support.
The author thanks Hiroki Matui for his advices in Lemma \ref{lem:5.5} and suggestions.
This work was also supported by JSPS KAKENHI Grant Number 15K04896.

%%%%%%%%%%%%%%%%%%%%%%%%%%%%%%

\end{document}